\RequirePackage{fix-cm}
\documentclass{svjour3}
\smartqed

\usepackage{graphicx}
\usepackage{amssymb}
\usepackage{amsmath}
\usepackage{subcaption}
\usepackage{booktabs}
\usepackage{multirow}
\usepackage{makecell}
\usepackage{enumitem}
\usepackage{multicol}
\usepackage{algorithm}
\usepackage{algorithmic}
\usepackage{pdflscape}
\usepackage[hidelinks, colorlinks=true, linkcolor=blue, citecolor=blue, urlcolor=magenta]{hyperref}
\setlist[enumerate]{label=(\roman*)}
\def\N{\mathcal{N}}
\def\R{\mathbb{R}}
\def\J{\mathcal{J}}
\def\S{\mathcal{S}}
\def\L{L_{\rm KKT}}
\def\grad{\nabla}
\def\T{\top}
\def\lam{\lambda}
\def\im{ \mathrm{im} }
\def\I{\mathcal{I}}
\def\tlam{\tilde{\lambda}}
\def\bx{\bar{x}}

\def\tx{\tilde{x}}
\def\dx{\Delta x}
\def\dz{\Delta z}

\def\blam{\bar{\lambda}}
\def\dlam{\Delta \lambda}

\def\tlam{\tilde{\lambda}}
\def\bp{\bar{p}}
\def\bq{\bar{q}}

\def\hC{\hat{C}}

\def\th{\tilde{h}}
\def\tM{\tilde{M}}
\def\tz{\tilde{z}}
\def\hnu{\hat{\nu}}

\def\pr{\mathrm{pr}}
\def\np{\mathrm{np}}
\def\Proj{\mathrm{\Pi}}
\journalname{}

\begin{document}

\title{A sequential linear complementarity problem method for generalized Nash equilibrium problems%\thanks{Grants or other notes
}
% \subtitle{Do you have a subtitle?\\ If so, write it here}

\titlerunning{A sequential LCP method for GNEPs}        % if too long for running head

\author{ \mbox{ Ruoyu Diao$^{1, 2}$ \and 
        Yu-Hong Dai$^{1, 2}$ \and
        Liwei Zhang$^{3, 4, *}$} %etc.
}

\authorrunning{R. Diao, Y.H.~Dai, L. Zhang} % if too long for running head

\institute{Ruoyu Diao \at
           \email{diaoruoyu18@mails.ucas.ac.cn}           %  \\
           \and
           Yu-Hong Dai \at
          \email{dyh@lsec.cc.ac.cn}
          \and 
          Liwei Zhang \at
          \email{Zhanglw@mail.neu.edu.cn} \at
\at $^*$ Corresponding author 
\at $^{1}$ State Key Laboratory of Mathematical Sciences, Academy of Mathematics and Systems Science, Chinese Academy of Sciences, Beijing 100190, China
\at $^{2}$ School of Mathematical Sciences, University of Chinese Academy of Sciences, Beijing 100049, China
\at $^{3}$ National Frontiers Science Center for Industrial Intelligence and Systems Optimization, Northeastern University, Shenyang 110819, China
\at $^{4}$ Key Laboratory of Data Analytics and Optimization for Smart Industry (Northeastern University), Ministry of Education, Shenyang 110819, China
}

\maketitle

\begin{abstract}

Generalized Nash equilibrium problems (GNEPs) arise in various applications where multiple players minimize individual cost functions subject to coupled constraints. 
A relatively unexplored approach to solving such problems is via a sequence of (mixed) linear complementarity problems (LCPs). Compared with the nonlinear equilibrium subproblems arising in recently popular penalty-based methods such as augmented Lagrangian methods, these LCPs are often substantially easier to solve. However, the existing literature on this approach is very limited, largely because of the difficulty of assessing the search directions generated by the subproblems and establishing a principled step-length acceptance criterion. This paper proposes a sequential linear complementarity problem (SLCP) method with a comprehensive convergence analysis. To assess the search directions, we introduce a novel merit function analogous to the $\ell_1$ penalty function in sequential quadratic programming. The merit function is shown to decrease along the search directions generated by the subproblems under suitable assumptions, thereby guaranteeing the global convergence of the SLCP method. We further establish local quadratic convergence and analyze the solvability of the subproblems. Preliminary numerical results demonstrate the effectiveness and competitiveness of the proposed method relative to existing approaches.

\end{abstract}

\section{Introduction}
Consider a generalized Nash equilibrium problem (GNEP) involving $N$ players. Let $x^\nu \in \R^{n_\nu}$ denote the strategy of player $\nu \in [N]:=\{ 1,...,N\}$, and define $x:=(x^\nu)_{\nu=1}^N \in \R^n$ as the concatenation of all players' strategies, where $n:=\sum_{\nu\in [N] } n_\nu$. Given $x^{-\nu}:=(x^\mu)_{\mu \neq \nu}$, player $\nu$ solves the optimization problem:
\begin{equation}\label{GNEP}
    \begin{aligned}
         \min\limits_{x^\nu \in \R^{n_\nu}}\ \ & \theta^\nu (x^\nu,x^{-\nu})\\
         {\rm s. t.}\ \ \ \ & g_i^\nu (x^\nu, x^{-\nu}) \le 0 \quad {\rm for}\ {\rm all}\ i=1,...,m_\nu.\\
    \end{aligned}
\end{equation}
Here, $\theta^\nu : \R^{n} \to \R$ is the cost function of player $\nu$, and $g^\nu_i : \R^{n} \rightarrow \R$ is the $i$-th constraint of player $\nu$. Assume that for each player $\nu$, the functions $\theta^\nu$ and $g^\nu_i$ are twice continuously differentiable on $\R^n$. Let $m:= \sum_{\nu\in [N]} m_\nu$ denote the total number of constraints. If each constraint function $g^\nu_i$ depends only on the individual strategy $x^\nu$, then the GNEP reduces to a standard Nash equilibrium problem (NEP). 

A vector $\bx= (\bx^1,...,\bx^N)$ is called a generalized Nash equilibrium (GNE) if 
\begin{equation}\label{def:nash equilibrium}
    \bx^\nu \in \arg\min\limits_{x^\nu} \left\{ \theta^\nu (x^\nu,\bx^{-\nu}) \mid g^\nu_i (x^\nu,\bx^{-\nu}) \le 0,\ i=1,...,m_\nu   \right\},\quad \forall \nu\in [N].
\end{equation}
When the functions $\theta^\nu(\cdot,x^{-\nu})$ and $g^\nu_i (\cdot,x^{-\nu})$ are convex for all $x^{-\nu}$, the GNEP is called \emph{player-convex}. An important subclass of GNEPs is the \emph{jointly-convex} GNEP, where 
$m_1=...=m_N$, $g^1_i = ... = g^N_i:=g_i$, and $g_i (x)$ is convex in $x$ for all $i\in[m_1]$.

The GNEP was first introduced in the seminal works of Arrow and Debreu~\cite{arrow1954existence,debreu1952social} as a rigorous framework for modeling social and economic equilibria. Since then, GNEPs have found wide applications in transportation systems~\cite{ban2019general,stein2018noncooperative,xiao2007competition,zhou2005generalized}, wireless communication networks~\cite{kesselman2005game,pang2008distributed}, environmental pollution control~\cite{breton2006game,krawczyk2005coupled}, and various related fields. We refer the reader to~\cite{facchinei2007generalized,palomar2010convex} for
comprehensive surveys on this active research area and its diverse applications.

Motivated by these broad applications, a rich variety of numerical methods has been developed for GNEPs, including Newton-type methods~\cite{dreves2014new,dreves2011solution,facchinei2009generalized,fischer2016globally}; see also the surveys~\cite{facchinei2007generalized,facchinei2010generalized,fischer2014generalized}. Among these, penalty-based methods have attracted considerable attention in recent years. By penalizing the coupled constraints, they reformulate a GNEP as a standard NEP. Representative approaches include exact penalty methods~\cite{facchinei2010penalty,facchinei2011partial,fukushima2011restricted,pang2005quasi}, which establish exact penalization results under suitable assumptions but suffer from nonsmooth Nash equilibrium subproblems. To overcome this difficulty, Kanzow and Steck~\cite{kanzow2016augmented,kanzow2019quasi} proposed augmented Lagrangian methods (ALMs) and established global convergence under constraint qualifications tailored to GNEPs. This framework was subsequently extended to infinite-dimensional settings in~\cite{kanzow2019multiplier}, and a first-order iteration complexity bound was recently derived in~\cite{jordan2023first}. Despite their appealing global convergence theory, ALMs still require solving Nash equilibrium subproblems whose objective functions are not twice continuously differentiable. As a result, these subproblems may be difficult to solve efficiently and robustly in practice. Very recently, Wang et al.~\cite{wang2025symmetric} addressed this issue for a special class of jointly-convex GNEPs with linear constraints in Hilbert spaces by proposing a symmetric Gauss--Seidel based majorized ALM, in which each Nash equilibrium subproblem is decomposed into a sequence of quadratic programs. 

In contrast, linear complementarity problems (LCPs) are typically much easier to solve than the nonlinear equilibrium subproblems arising in ALMs. In fact, the idea of solving equilibrium problems through a sequence of LCPs dates back to the early works 
\cite{mathiesen1985computation,mathiesen1985computational,mathiesen1987algorithm}, where efficient numerical methods for solving the LCPs were already discussed. However, these early works did not provide a comprehensive analysis of the algorithmic framework, in particular with respect to convergence and the solvability of the subproblems. Since then, this line of research has received only limited attention. A central obstacle is the lack of suitable merit functions for evaluating the search directions generated by the linear complementarity subproblems, which is crucial for the design of globally convergent algorithms. Moreover, exact characterizations of the conditions guaranteeing the local quadratic convergence remain far from fully understood. In this regard, Schiro, Pang, and Shanbhag \cite{schiro2013solution} noted that \textit{``The application of this methodology to nonlinear equilibrium problems through a sequence of linearized problems may also be a fruitful area of research''}.

Against this background, we propose a sequential linear complementarity problem (SLCP) method for solving GNEPs and develop a thorough convergence analysis for it. We linearize the smooth part of the Karush--Kuhn--Tucker (KKT) system at each iteration and thereby obtain a linear complementarity problem. To assess the search directions generated by these LCPs, we introduce a novel merit function analogous to the $\ell_1$ penalty function in sequential quadratic programming (SQP) methods. Based on this merit function, we employ an inexact line search procedure to ensure global convergence. From a methodological perspective, the proposed framework can be viewed as a multi-player analogue of SQP methods. More specifically, our contributions are as follows: 
\begin{enumerate}
    \item \textbf{An SLCP method with global convergence.} We introduce a novel merit function tailored to the structure of GNEPs. Under conditions closely analogous to those of the classical SQP method with exact Lagrangian Hessians \cite{han1977globally} for nonlinear programming (NLP), we establish global convergence of the SLCP method.
    \item \textbf{Local quadratic convergence.} We derive sufficient conditions for the local quadratic convergence of the method. Our analysis reveals that semistability and hemistability are not equivalent at local Nash equilibria of GNEPs. We also derive an error bound under semistability.
    \item \textbf{Solvability of the subproblems.} For the subproblems arising in the SLCP method, we provide a thorough solvability analysis, identifying conditions under which the subproblems are guaranteed to admit solutions.
    \item \textbf{Numerical validation.} Preliminary numerical experiments  comparing the proposed SLCP method with the ALM \cite{kanzow2016augmented} and Newton-type methods \cite{dreves2011solution} are presented. The results demonstrate the promising performance of the SLCP method in both accuracy and computational efficiency.
\end{enumerate}

\subsection{Organization}

The remainder of the paper is organized as follows.  Section~\ref{sec:merit function and descent direction} introduces the novel merit function and investigates its key properties. Section~\ref{sec:alg} presents the SLCP method, including the algorithmic framework and the solvability of the subproblems. Global convergence is established in Section~\ref{sec:global-convergence}, and local quadratic convergence is analyzed in Section~\ref{sec:local convergence}.  Numerical results comparing different approaches are reported in Section~\ref{sec:experiments}. Section~\ref{sec:conclusions} concludes the paper. Appendix~\ref{appendix:numerical results} contains detailed numerical results reported in Section~\ref{sec:experiments}. Appendix~\ref{appendix:further analysis for the internet switching model} contains additional theoretical analysis of the internet switching model introduced in Section~\ref{sec:merit function and descent direction}, including the solvability of the subproblems and the strong regularity of the KKT system.

\subsection{Notation}

Throughout the paper, let $\N_{\R^m_+}(y)$ denote the normal cone \cite[Definition 6.3]{rockafellar1998variational} to $\R^m_+$ at $y$. The partial Jacobian of a continuously differentiable function $f:\R^n\rightarrow \R^m$ with respect to $x^\nu$ at $x$ is denoted by $\J_{x^\nu} f(x)$, and the corresponding gradient with respect to $x^\nu$ at $x$ is denoted by $\grad_{x^\nu} f(x):= (\J_{x^\nu} f(x) )^\top$. For a directionally differentiable function $\phi:\R^n \rightarrow \R$, 
we denote its directional derivative along the vector $p\in \R^n$ at $x$ by $\phi'(x;p)$. The notation $(x)_+$ represents $\left(\max\left\{ x_i, 0 \right\}\right)_{i=1}^n$. For a matrix $M\in \R^{m\times n}$, its image is denoted by $\im (M)$. The unit ball in $\R^n$ is denoted by $\mathbb{B}$. Let $\mathbf{I}_n$ be the $n\times n$ identity matrix, and let $\mathbf{0}_{n\times m}$ be the $n\times m$ zero matrix. Additional notation will be introduced as needed.

\section{A merit function and its descent property for GNEPs}
\label{sec:merit function and descent direction}

In this section, we construct a merit function tailored to the structure of GNEPs and analyze its descent property.
We begin by recalling the KKT system associated with a general GNEP and show that linearizing this system naturally yields a GNE subproblem.

Let $\bx$ be a GNE of the GNEP \eqref{GNEP}. Suppose that standard  constraint qualifications (e.g., Mangasarian--Fromovitz constraint qualification) hold. For each $\nu \in [N]$, there exists a vector of Lagrange multipliers $\lam^\nu:=( \lam^\nu_1,...,\lam^\nu_{m_\nu})$ such that the following KKT conditions hold:
\begin{equation}\label{KKT-each-player}
    \begin{array}{c}
       \nabla_{x^\nu} \theta^\nu (\bx^{\nu},\bx^{-\nu}) + \sum\limits_{i=1}^{m_\nu} \lam^\nu_i \nabla_{x^\nu} g^\nu_i (\bx^\nu,\bx^{-\nu}) = 0,\\
       \lam^\nu_i \ge 0, \quad g^\nu_i (\bx^\nu,\bx^{-\nu}) \le 0, \quad \lam^\nu_i g^\nu_i(\bx^\nu,\bx^{-\nu}) = 0,\ i=1,...,m_\nu.
    \end{array}
\end{equation}
Conversely, if the GNEP is player-convex and $\bx$ satisfies \eqref{KKT-each-player} with a multiplier $\lam:=(\lam^1,...,\lam^N)$, $\bx$ is a GNE of \eqref{GNEP}.

For each player $\nu$, the vector of constraints is denoted by $g^\nu (x):= (g^\nu_i (x))_{i=1}^{m_\nu}$. Let $ L^\nu(x,\lambda^\nu):= \theta^\nu (x) + (\lambda^\nu)^\top g^\nu(x)$ be the Lagrangian function of player $\nu$. Let $G(x):= (g^\nu(x))_{\nu=1}^{N}$ and $F(x,\lam):= ( \nabla_{x^\nu} L^\nu(x,\lambda^\nu))_{\nu=1}^N$. The concatenated KKT system of the GNEP \eqref{GNEP} can be written compactly as 
\begin{equation}\label{KKT}
    F(x,\lambda) = 0, \quad 0\le \lambda \perp -G(x) \ge 0,
\end{equation}
where $0\le \lambda \perp -G(x) \ge 0$ denotes the collection of equations 
$$
\lambda\in \R^m_+,\ G(x) \in \R^m_-,\ \text{and } \lam^\nu_i g^\nu_i (x) = 0,\  \forall i\in [m_\nu], \forall \nu \in [N].
$$
A pair $(\bx,\blam)$ satisfying \eqref{KKT} is called a KKT pair of \eqref{GNEP}.

By linearizing the KKT system~\eqref{KKT} at a given point $(x,\lam)$, we obtain the following (mixed) linear complementarity (LC) subproblem 
\begin{equation}\label{KKT-linearized}
    \begin{array}{c}
          F(x,\lam) + \J_x F(x,\lam) p + E(x) q = 0,  \\
          0 \le \lam + q \perp -G(x) - \J_x G(x) p \ge 0,
    \end{array}
\end{equation}
where 
 \begin{equation*}
    E(x) = 
    \begin{pmatrix}
    (\J_{x^1} g^1(x))^\T &  & 0\\
     & \ddots & \\
     0 & & (\J_{x^N} g^N(x))^\T
    \end{pmatrix}, \ 
     \J_{x^\nu} g^\nu (x) \in \R^{m_\nu \times n_\nu}.
\end{equation*}
The LC subproblem~\eqref{KKT-linearized} is equivalent to the following GNE subproblem at $(x,\lam)$ under convexity, where each player $\nu$ solves 
\begin{equation}\label{subproblem}
    \begin{aligned}
         \min\limits_{p^\nu \in \R^{n_\nu}}\ \ & \grad_{x^\nu} \theta^\nu (x)^\top p^\nu  + \frac{1}{2}(p^\nu)^\top \grad_{x^\nu x^\nu} L^\nu(x,\lam^\nu)p^\nu \\
         & \qquad \qquad \qquad \qquad \qquad + (p^\nu)^\top \grad_{x^\nu x^{-\nu}}L^\nu (x,\lam^\nu)p^{-\nu}  \\
         {\rm s. t.}\ \ \ \ & g_i^\nu (x) + \grad_{x} g^\nu_i (x)^\top p \le 0 \quad {\rm for}\ {\rm all}\ i=1,...,m_\nu.\\
    \end{aligned}
\end{equation}
This subproblem is an affine GNEP. When $N=1$, it reduces to a quadratic programming problem, which coincides with the subproblem derived by the classical SQP method \cite{han1977globally}. When the GNEP \eqref{GNEP} is player-convex, each solution $\bp$ of the GNE subproblem~\eqref{subproblem} at $(x,\lambda)$ admits a multiplier $\blam$ such that $(\bp,\blam - \lambda)$ solves the LC subproblem~\eqref{KKT-linearized} at $(x,\lam)$. 

To better illustrate the structure of \eqref{subproblem}, we present a practically motivated example, 
which will be analyzed in detail in Appendix~\ref{appendix:further analysis for the internet switching model}. This example also serves as one of the test problems in our numerical experiments.

\begin{example}\label{exm-internet}
This is a variant of the internet switching model introduced by \cite{kesselman2005game} and further analyzed by \cite{facchinei2009generalized}. It considers a setting with  $N$ selfish users competing for buffer capacity in a shared network. Each user has a data transmission requirement $x^\nu \in [l_\nu, L_\nu]$, where $0<l_\nu\le  L_\nu$ and $L_\nu$ may be $+\infty$. Let $B>0$ be the maximum buffer capacity of the network. Some users are constrained by this buffer capacity, while others are not. Accordingly, let $[N_1] \subset [N]$ be the set of the constrained users subject to the buffer capacity, and $[N] \setminus [N_1]$ be the set of unconstrained ones. Each user solves the following optimization problem:
\begin{equation}\label{exm-1}
    \begin{aligned}
		\min\limits_{x^\nu \in \R} \ \ &\ -\frac{x^\nu}{\sum_{\mu \in [N]} x^\mu} \left(1- \frac{\sum_{\mu\in [N]} x^\mu}{B}\right)\\
		\text{s.t.}\ \ \ \ &\  l_\nu \le x^\nu \le L_\nu,\\
		& \sum\limits_{\mu \in [N]} x^\mu \le B \quad {\rm if}\ \nu \in [N_1],
	\end{aligned}
\end{equation}
where the term $\frac{x^\nu}{\sum_{\mu \in [N]} x^\mu}$ denotes the transmission rate, and $1- \frac{\sum_{\mu\in [N]} x^\mu}{B}$ represents the congestion level. This is a player-convex GNEP. When $[N_1] = [N]$, the model reduces to a jointly-convex GNEP. At a given point $(x,\lam)$, the GNE subproblem associated with \eqref{exm-1} is 
\begin{equation*}
    \begin{aligned}
         \min\limits_{p^\nu \in \R}\ \ &  \left(\frac{1}{B} - \frac{S - x^\nu}{S^2}\right)p^\nu + \frac{S-x^\nu}{S^3}(p^\nu)^2 + \sum\limits_{\mu \in [N]\setminus \{\nu\} }\frac{S- 2x^\nu }{S^3}p^\nu p^\mu  \\
         {\rm s. t.}\ \ \ \ & l_\nu \le x^\nu + p^\nu \le L_\nu,\\
         & \sum\limits_{\mu \in [N]} (x^\mu + p^\mu) \le B\quad {\rm if}\ \nu \in [N_1],
    \end{aligned}
\end{equation*}
where $S:= \sum_{\mu \in [N]} x^\mu$. 
\end{example}

A suitable merit function is crucial for the design of globally convergent algorithms. However, unlike NLP, GNEPs cannot directly rely on the individual players' objectives to assess whether a computed descent direction makes the current iterate closer to a GNE. The following example illustrates this issue.

\begin{example}
Consider the two-player GNEP:
\begin{equation}\label{descent direction of example}
	\begin{array}{llllll}
		\min\limits_{x^1 \in \R} & \frac{1}{2} (x^1)^2 + x^1 x^2,  & \qquad \qquad  & \min\limits_{x^2 \in \R} & \frac{1}{2} (x^2)^2 + 2x^1 x^2.
	\end{array}
\end{equation}
\end{example}
The unique GNE of this game is $(\bx^1,\bx^2) = (0,0)$. Suppose that the initial point is $(x^{1},x^{2}) = (1,-1.5)$. If one requires both objective values to decrease at each iteration, the iterates would be confined to the region where each player's objective is below its initial value. Since the equilibrium $(0,0)$ yields objective values $(0,0)$, which are higher than the initial objectives $(-1,-15/8)$, such a requirement would prevent the algorithm from approaching the equilibrium. This demonstrates that using only the objective values as a merit function is inappropriate.

The above observation motivates the design of an SLCP framework in which the merit function accounts for the first-order optimality conditions of all players' optimization problems. The merit function is analogous to the $\ell_1$ penalty function in SQP methods. To construct such a function, we recall some basic concepts and properties of linear complementarity problems. 

% With these preliminaries, we then construct an appropriate merit function and establish conditions under which the search direction obtained from the LC subproblem is guaranteed to be a descent direction. Moreover, we show how these conditions reduce in the cases of NEPs and nonlinear programming problems, thereby illustrating that our algorithm can be viewed as a natural generalization of the classical SQP with exact Lagrangian Hessians to GNEPs.

Given a vector $h\in \R^n$ and a matrix $M\in \R^{n\times n}$, a standard linear complementarity problem \cite{cottle2009linear}, denoted by ${\rm LCP}(h,M)$, is to find a vector $z\in \R^n$ such that 
\begin{equation}\label{def-LCP}
    0 \le z \perp Mz + h \ge 0.
\end{equation}
The solution set of \eqref{def-LCP} at $(h,M)$ is denoted by ${\rm SOL}(h,M)$. The matrix $M$ (not necessarily symmetric) is said to be positive semidefinite if $z^\top M z \ge 0$ for all $z\in \R^n$.

The next lemma provides a sufficient condition 
for the solvability of ${\rm LCP}(h,M)$.

\begin{lemma}\label{lemma:LCP-solvability}
    \cite[Theorem 3.8.6]{cottle2009linear} Let $M\in \R^{n\times n}$ be positive semidefinite and let $h\in \R^n$ be given. If the implication 
    \begin{equation*}
        u\ge 0,\ M u \ge 0,\ u^\top M u = 0 \implies u^\top h \ge 0
    \end{equation*}
    is valid, then {\rm LCP}$(h,M)$ is solvable.
\end{lemma}

We present a stability property of LCP$(h,M)$ when $M$ is positive semidefinite, which is a consequence of \cite[Theorem 7.5.1]{cottle2009linear}.
\begin{lemma}\label{lemma:LCP stability}
    Let $M\in \R^{n\times n}$ be a positive semidefinite matrix, and let $h$ satisfy the following implication:
    \begin{equation*}
        0\neq u \ge 0,\ M u \ge 0,\ u^\top M u =0 \implies u^\top h >0.
    \end{equation*}
    Then there exist positive constants $c$, $\epsilon$ and $L$ such that, for all $(\th,\tM) \in \R^{n} \times \R^{n\times n}$ with $\Vert h - \tilde{h} \Vert + \Vert M -\tilde{M}\Vert  \le \epsilon$, the following statements hold:
    \begin{enumerate}
        \item {\rm LCP}$(\th,\tM)$ is solvable;
        \item For all  $\tz \in {\rm SOL}(\th,\tM)$, $\Vert \tz \Vert \le c$;
        \item ${\rm SOL}(\th,\tM) \subset {\rm SOL}(h,M) + L (\Vert h - \tilde{h} \Vert + \Vert M - \tilde{M}\Vert ) \mathbb{B}$.
    \end{enumerate}
\end{lemma}

We now introduce the merit function associated with the GNEP. Define $\Phi_\rho : \R^n\times \R^m \rightarrow \R$ by
 \begin{equation}\label{def:merit function}
    \Phi_\rho(x,\lam):= \left( - \lam^\top G(x) \right)_+
    + \frac{\rho}{2} \Vert F(x,\lambda) \Vert^2 + \sum\limits_{\nu \in [N]} \sum\limits_{i\in[m_\nu]} \left( g^\nu_i ({x})\right)_+.
\end{equation}
The merit function $\Phi_\rho$ admits a natural interpretation. Consider the following optimization problem:
 \begin{equation}\label{def:merit function from}
        \begin{aligned}
         \min\limits_{x\in \R^n,\lambda\in \R^m_+}\ \ & \left( - \lam^\top G(x) \right)_+ + \sum\limits_{\nu \in [N]} \sum\limits_{i\in[m_\nu]} \left( g^\nu_i ({x})\right)_+  \\
         {\rm s. t.}\ \ \ \ & F(x,\lam) = 0.
    \end{aligned}
\end{equation}
It is easy to verify that any KKT pair satisfying \eqref{KKT} is a global minimizer of \eqref{def:merit function from}, and that $\Phi_\rho$ can be viewed as the $\ell_2$ penalty function associated with \eqref{def:merit function from}. However, this observation is mainly conceptual. The problem \eqref{def:merit function from} is generally nonconvex, and there is no known algorithm that can guarantee convergence to a global minimizer. Consequently, the formulation~\eqref{def:merit function from} cannot be used directly to solve the GNEP.

The next proposition shows that $\Phi_\rho$ serves as a proper merit function for the GNEP~\eqref{GNEP}.

\begin{proposition}\label{prop:Phi = 0 implies KKT}
 For any $\rho>0$, a pair $(x,\lam)$ is a KKT pair of \eqref{GNEP} if and only if
 \begin{equation}
     \Phi_\rho(x,\lam) = 0\ \text{and}\ \lam \in \R^m_+.
 \end{equation}
\end{proposition}
\begin{proof}
It suffices to prove the ``if'' part. Suppose that $\Phi_\rho (x,\lam) = 0$ and $\lambda \in \R^m_+$. Then,
\begin{equation*}
    F(x,\lam) = 0, \quad 
    \lam^\top G(x) \ge 0, \quad 
    G(x) \le 0,\quad  \lam \ge 0,
\end{equation*}
which implies 
\begin{equation*}
    F(x,\lam) = 0,\quad
    0\le \lam \perp - G(x) \ge 0. 
\end{equation*}
This completes the proof.
\end{proof}

We now introduce a key condition for the KKT system of a GNEP, which guarantees the descent of the merit function along the direction obtained by solving the LC subproblem.

\begin{definition}
Let $\alpha, \beta >0$ be given. The KKT system~\eqref{KKT} is said to satisfy the \emph{$(\alpha,\beta)$-monotonicity condition} at $(x,\lam)$ if the following hold:
\begin{enumerate}
    \item The matrix $\J_{x} F(x,\lambda)$ is nonsingular with $\| \J_x F(x,\lambda)^{-1} \| \le \alpha$;
    \item For any $u \in \R^m$,
    \begin{equation}\label{monotone-property}
         u^\top \J_{x} G(x)\J_{x} F(x,\lambda)^{-1}E(x) u
        \ge \beta \Vert \J_{x} G(x)^\top u \Vert^2.
    \end{equation}
\end{enumerate}
\end{definition}

\begin{remark}\label{rmk:equivalence-monotone}
The $(\alpha,\beta)$-monotonicity condition specializes to well-known conditions when the GNEP \eqref{GNEP} reduces to different cases. For example, if \eqref{GNEP} reduces to an NEP, then $\J_x G(x)^\top = E(x)$. In this case, condition (ii) is implied by the $\beta$-cocoercivity of $F(\cdot,\lam): \R^n \to \R^n$. 
If \eqref{GNEP} further reduces to an NLP problem, then the $(\alpha,\beta)$-monotonicity condition becomes
\begin{equation*}
    \frac{1}{\alpha} \Vert y \Vert^2 \le y^\top\grad^2_{x x}L(x,\lam) y  \le \frac{1}{\beta} \Vert y \Vert^2 \quad \text{for all}\ y\in \im (E(x)), 
\end{equation*}
and $\nabla_{xx}^2 L(x,\lam)$ is nonsingular.  This condition is slightly weaker than the classical global convergence assumptions for the early SQP method \cite{han1977globally} with exact Lagrangian Hessians, where the quasi-Newton matrix $H_k$ in \cite[Theorem 3.2]{han1977globally} is replaced by  $\grad^2_{x x}L(x,\lam)$ and $y$ ranges over the entire space.
\end{remark}

% \vspace{12pt}

For notational convenience, we introduce several index sets at a given point $x\in \R^n$. For each player $\nu\in[N]$, define 
\begin{equation}
\begin{aligned}
J^\nu_+ := \left\{ i\in [m_\nu] \mid  g^\nu_i (x) >0 \right\},\\
    J^\nu_0 := \left\{ i\in [m_\nu] \mid g^\nu_i (x)=0  \right\},\\
    J^\nu_- := \left\{ i\in [m_\nu] \mid g^\nu_i (x)<0  \right\}.
\end{aligned}
\end{equation}
The next theorem shows that the merit function $\Phi_\rho$ decreases along the direction computed by solving the LC subproblem.

\begin{theorem}\label{thm:descent}
	Suppose that the KKT system \eqref{KKT} satisfies the
	$(\alpha,\beta)$-monotonicity condition at $(x,\lambda)\in\R^n\times \R^m_+$ for some $\alpha,\beta>0$. Let $(p,q)$ be a solution to 
    the LC subproblem~\eqref{KKT-linearized}. Then for any $\rho\ge \frac{2 \alpha^2}{\beta} +1 $, the directional derivative of $\Phi_\rho$ along $(p,q)$ at $(x,\lam)$ satisfies
    $$
        \Phi^\prime_\rho \bigl( (x,\lambda); (p, q) \bigr) \le -   \Phi_\rho ( x,\lambda) .
    $$
    Moreover, 
    $$
    \Phi^\prime_\rho \bigl( (x,\lambda); (p, q) \bigr) = 0\quad \text{if and only if}\quad \Phi_\rho ({x},{\lambda}) = 0.
    $$
\end{theorem}

\begin{proof}
Let $\phi(x,\lam) := \left(-\lam^\top G(x) \right)_+$. The directional derivative of $\Phi_\rho$ along $(p,q)$ at $(x,\lam)$ is
\begin{equation}\label{eq:Phi'}
    \begin{aligned}
    \Phi'_\rho \bigl( (x,\lam);(p,q) \bigr) & = \phi'\bigl((x,\lam);(p,q) \bigr) - \rho \Vert  F(x,\lam) \Vert^2 \\
    & \qquad \qquad +\sum\limits_{\nu\in[N]} \left( \sum\limits_{i\in J^\nu_0} \left(\J_x g^\nu_i (x) p \right)_+ + \sum\limits_{i\in J^\nu_+} \J_x g^\nu_i (x) p \right).
    \end{aligned}
\end{equation}
Since $p$ is feasible,
\begin{equation*}
    g^\nu_i (x) + \J_x g^\nu_i(x) p \le 0 \quad \text{for all}\ \nu\in[N], i\in [m_\nu].
\end{equation*}
Thus, 
\begin{equation}\label{JG<0}
    \left\{\begin{array}{ll}
           \J_x g^\nu_i (x)p \le 0 &\quad \text{if}\ \nu\in [N], i\in J^\nu_0; \\
          \J_x g^\nu_i (x)p \le -g^\nu_i(x)<0 &\quad \text{if}\ \nu\in [N], i\in J^\nu_+.
    \end{array}\right.
\end{equation}
To proceed, we distinguish three cases according to the value of $\lambda^\top G(x)$.

\textit{Case 1:} $\lam^\top G(x) <0$. In this case, 
\begin{equation}\label{eq:phi'-case-1}
\phi' \bigl((x,\lam);(p,q) \bigr) = -q^\top G(x) - \lam^\top \J_x G(x)p.
\end{equation}
Since $(p,q)$ solves the LC subproblem~\eqref{KKT-linearized}, 
\begin{equation}\label{complementarity-equation}
\begin{aligned}
           \lam^\top \J_x G(x)p & = - \lam^\top G(x) - q^\top \J_x G(x)p - q^\top G(x), \\
      p &= -\J_x F(x,\lam)^{-1}\left( F(x,\lam) + E(x)q  \right).
\end{aligned}
\end{equation}
Combining \eqref{eq:phi'-case-1} and \eqref{complementarity-equation}  yields 
\begin{equation}\label{phi-prime}
    \begin{aligned}
    \phi' \bigl( (x,\lam);(p,q) \bigr) & = \lam^\top G(x) + q^\top \J_x G(x) p \\
    & = \lam^\top G(x) - q^\top \J_x G(x)\J_x F(x,\lam)^{-1}\left( F(x,\lam) + E(x)q  \right). \\
    \end{aligned}
\end{equation}
Due to the $(\alpha,\beta)$-monotonicity condition, for any $\delta>0$,
\begin{equation}
    \begin{aligned}
    \phi' \bigl( (x,\lam);(p,q) \bigr)     &   \le \lam^\top G(x) -\beta \Vert \J_x G(x)^\top q \Vert^2 + \frac{\delta}{2}\Vert \J_x G(x)^\top q \Vert^2 \\
    & \qquad \qquad + \frac{1}{2\delta}\Vert \J_x F(x,\lam)^{-1} \Vert^2 \Vert F(x,\lam) \Vert^2\\
    & \le \lam^\top G(x) + \left( \frac{\delta}{2} - \beta \right) \Vert \J_x G(x)^\top q \Vert^2 + \frac{\alpha^2}{2 \delta} \Vert F(x,\lam) \Vert^2,
    \end{aligned}
\end{equation}
Choosing $\delta = \frac{\beta}{2}$ and $\rho \ge \frac{2\alpha^2}{\beta} + 1$, we can get by \eqref{eq:Phi'}, \eqref{JG<0} and \eqref{phi-prime} that
\begin{equation}\label{Phi-descent}
\begin{aligned}
   \Phi'_\rho \bigl( (x,\lam);(p,q) \bigr) & \le  \left( \frac{\delta}{2} - \beta \right)\Vert \J_x G(x)^\top q \Vert^2 + \left( \frac{\alpha^2}{ \beta} - \rho \right) \Vert F(x,\lam)\Vert^2 \\
   & \qquad \qquad \qquad \qquad  - \left(-\lam^\top G(x) \right)_+   - \sum\limits_{\nu \in [N] } \sum\limits_{i\in J^\nu_+} g^\nu_i (x) \\
   &\le -  \Phi_\rho (x,\lam).
\end{aligned}
\end{equation}

\textit{Case 2:} $\lam^\top G(x) = 0$. In this case,
\begin{equation*}
\phi' \bigl((x,\lam);(p,q) \bigr) = \left(-q^\top G(x) - \lam^\top \J_x G(x)p \right)_+.
\end{equation*}
If $q^\top G(x) + \lam^\top \J_x G(x)p\ge 0$, then $\phi' \bigl((x,\lam);(p,q) \bigr)\le 0$, which together with~\eqref{JG<0} implies 
$$
\Phi'_\rho \bigl( (x,\lam);(p,q) \bigr)\le -\rho \Vert F(x,\lam) \Vert^2  - \sum\limits_{\nu \in [N]} \sum\limits_{i\in J^\nu_+}  g^\nu_i (x)\le -  \Phi_\rho (x,\lam).
$$ 
Otherwise, the proof is analogous to \textit{Case 1}. 

\textit{Case 3:} $\lam^\top G(x)>0$.  Then $\phi' \bigl((x,\lam);(p,q) \bigr) = 0$, and the descent property follows immediately.

Finally, we verify that $\Phi'_\rho \bigl( (x,\lam);(p,q) \bigr) = 0$ if and only if the merit function $\Phi_\rho (x,\lam) = 0$. Suppose that $\Phi'_\rho \bigl( (x,\lam);(p,q) \bigr) = 0$. The previous analysis yields $\Phi_\rho (x,\lam) \le 0$. Since $\Phi_\rho (x,\lam) \ge 0$ by definition, we conclude $\Phi_\rho(x,\lam) = 0$. Conversely, suppose $\Phi_\rho (x,\lam) = 0$. Then Proposition~\ref{prop:Phi = 0 implies KKT} gives 
$$
F(x,\lam) = 0, \quad 0\le \lam \perp -G(x)\ge 0.
$$
Using \eqref{eq:Phi'}, \eqref{JG<0} and \eqref{phi-prime}, we have
% \begin{equation}\label{ineq-1}
%     \begin{aligned}
%     \Phi'_\rho \bigl( (x,\lam);(p,q) \bigr)  = (q^\top \J_x G(x)p)_+.
%     \end{aligned}
% \end{equation}
% Since $\J_x F(x,\lam)$ is nonsingular, the mixed LCP \eqref{KKT-linearized} implies
% \begin{equation}\label{p-expression}
%     p = -\left(\J_x F(x,\lam)\right)^{-1}E(x)q.
% \end{equation}
% Substituting \eqref{p-expression} into \eqref{ineq-1}, we obtain 
\begin{equation}\label{ineq-2}
    \begin{aligned}
    \Phi'_\rho \bigl( (x,\lam);(p,q) \bigr) = \left(- q^\top \J_x G(x) \J_x F(x,\lam)^{-1} E(x) q \right)_+.\\
    \end{aligned}
\end{equation}
By the $(\alpha,\beta)$-monotonicity condition, the right-hand side of~\eqref{ineq-2} vanishes, and hence $\Phi'_\rho \bigl( (x,\lam);(p,q) \bigr) = 0$. This concludes the proof.
\end{proof}

\section{A sequential linear complementarity problem method}
\label{sec:alg}

In this section, we present the SLCP method. Its construction is based on the merit function~\eqref{def:merit function}, which provides both a descent criterion and a foundation for the global convergence analysis.

The section is organized as follows. 
In Subsection~\ref{subsec:framework}, we present the overall algorithmic framework and establish the finite
termination of the inexact line search procedure. 
Subsection~\ref{subsec:subproblem-solvability} discusses the solvability of the LC subproblems, which is essential for practical implementation. 

% 	\begin{algorithm}[H]
% 		\caption{The basic algorithmic framework of SLCP for GNEPs}
% 		\begin{algorithmic}
% 			\STATE{\textbf{ Step 1: } Choose $x^{0}\in\mathbb{R}^n$, $\lambda^{0}\in\mathbb{R}^m$, $\rho >0$, 
% 			some constant $\delta>0$, and a sequence $\{\epsilon_k\}$ with $\sum_{k=1}^\infty \epsilon_k <\infty$. Let $k:=1$.}
% 			\STATE{\textbf{ Step 2: } Compute the LCP subproblem at $(x^k,\lambda^k)$ to obtain the direction $(p^k,q^k)$.}
			
% 			\STATE{\textbf{ Step 3: } If the stopping criteria is attained, stop the algorithm. Otherwise, go to Step 4.}
% 			\STATE{\textbf{ Step 4: } Find $\tau^k$ such that $({x}^{k+1}, {\lam}^{k+1}) 
% 			= ({x}^{k}, {\lam}^{k})+\tau^k({d}^k,{q}^k)$ and 
% 			$$
% 				\Phi_{\rho}({x}^{k+1}, {\lam}^{k+1})\le \min\limits_{0\le \tau^k \le \delta} 
% 				\Phi_\rho(({x}^{k}, \lam^k)+\tau^k({p}^k,{q}^k)) +\varepsilon_k
% 			$$}
% 			\STATE{Set $k:=k+1$ and go to Step 2.} 
			
% 		\end{algorithmic}\label{algslcp}
% 	\end{algorithm}

\subsection{Algorithmic framework}\label{subsec:framework}

We now present the basic algorithmic framework of the  SLCP method for GNEPs. The method follows a classical descent scheme based on the merit function $\Phi_\rho$, combined with an inexact line search to ensure global convergence.

\begin{algorithm}[H]
\caption{The basic algorithmic framework of SLCP for GNEPs}\label{algslcp}
\begin{algorithmic}
    \STATE{\textbf{Step 1:} Choose $x^{0}\in\mathbb{R}^n$, $\lambda^{0}\in\mathbb{R}^m_+$, $\rho>0$, 
           and parameters $0<\eta<1$, $0<\tau_0 \le 1$. \hspace*{3.6em} Set  $k := 0$.}
    \STATE \textbf{Step 2:} If the stopping criterion is satisfied, \textbf{stop}; otherwise, continue.
    \STATE\textbf{Step 3:} Solve the LC subproblem at $(x^k,\lambda^k)$ to obtain a direction $(p^k,q^k)$. Set $\tau^k = \tau_0$.

    \STATE \textbf{Step 4:} Carry out a line search. If
        \[
            \Phi_{\rho}(x^k + \tau^k p^k, \lambda^k + \tau^k q^k)
            \le  (1-\eta \tau^k )\Phi_{\rho}(x^k, \lambda^k),
        \]
    \hspace*{3.6em} then accept $\tau^k$; otherwise, set $\tau^k = \tau^k /2$ and repeat Step 4.
    \STATE \textbf{Step 5:} Update
        \[
            (x^{k+1}, \lambda^{k+1}) = (x^k,\lambda^k) + \tau^k (p^k, q^k).
        \]
        \hspace*{4em}Set $k = k+1$, and return to Step 2.
\end{algorithmic}
\end{algorithm}
	
Some comments are due. First, at each iteration $k \ge 0$, we have $\lam^k \in \R^m_+$. Since $(p^k,q^k)$ solves the LC subproblem~\eqref{KKT-linearized} at $(x^k,\lam^k)$, it follows that $\lam^k + q^k \in \R^m_+$. Hence, if $\lam^k \in \R^m_+$ and $ 0\le \tau^k \le 1$, then 
$$ 
\lam^{k+1} = \lam^k + \tau^k q^k \in \R^m_+.
$$
With the initial point $\lambda^0 \in \mathbb{R}^m_{+}$, the property holds for all iterations. 

Second, suppose that the $(\alpha,\beta)$-monotonicity condition holds at $(x^k,\lambda^k)$ for some $\alpha,\beta>0$. By Theorem~\ref{thm:descent}, either the computed direction is a sufficient descent direction or $\Phi_\rho(x^k,\lambda^k) = 0$. In the latter case, $(x^k,\lambda^k)$ is already a KKT pair of \eqref{GNEP} as $\lam^k \in \R^m_+$, and the algorithm terminates.

Our third comment concerns the $(\alpha,\beta)$-monotonicity condition. It guarantees both the finite termination of the inexact line search and the solvability of the LC subproblem. These results will be established in Lemma~\ref{lemma:linesearch terminate} and Theorem~\ref{thm:subproblem-solvability-1}, and are classical results for SQP methods in an NLP setting. As noted in Remark~\ref{rmk:equivalence-monotone}, when $\im(E(x)) = \R^n$, the $(\alpha,\beta)$-monotonicity condition aligns with the positive semidefiniteness of the Lagrangian Hessian in an NLP problem. Such positive semidefiniteness of a quadratic term ensures both the subproblem solvability and the finite termination of the inexact line search; see, e.g., \cite{burke2014sequential}. Hence, our approach extends the classical SQP method with exact Lagrangian Hessians to GNEPs.

The next lemma shows that, under the $(\alpha,\beta)$-monotonicity condition,
the line search procedure in Step~4 of Algorithm~\ref{algslcp} terminates finitely, which ensures the well-posedness of the algorithm.

\begin{lemma}\label{lemma:linesearch terminate}
    Suppose that the KKT system~\eqref{KKT} satisfies the $(\alpha,\beta)$-monotonicity condition at $(x^k,\lam^k)$ for some $\alpha,\beta >0$. Let $\rho \ge \frac{2\alpha^2}{\beta} + 1$. Then the line search procedure in Step 4 of Algorithm~\ref{algslcp} terminates finitely. 
\end{lemma}

\begin{proof}
We proceed by contradiction. Assume that there exists a sequence $\{ \tau^k_i \}$ with $\lim_{i\rightarrow \infty} \tau^k_i =0$ such that 
$$
\Phi_\rho (x^k + \tau^k_i p^k, \lam^k + \tau^k_i q^k) > (1-\eta \tau^k_i) \Phi_\rho (x^k,\lam^k), \quad \forall\, i\ge 0.
$$
Then it follows from the definition of directional derivative that 
\begin{equation}\label{Phi'-one-side}
\begin{aligned}
    \Phi^{'}_{\rho} \bigl( (x^k,\lam^k); (p^k,q^k) \bigr)    & = \lim\limits_{ i\rightarrow \infty }\frac{\Phi_{\rho}(x^k + \tau^k_i p^k, \lambda^k + \tau^k_i q^k) - \Phi_\rho (x^k,\lam^k)}{\tau^k_i}\\
    & \ge -\eta \Phi_{\rho}(x^k, \lambda^k).
\end{aligned}
\end{equation}
By Theorem~\ref{thm:descent}, we have
\begin{equation*}
    \Phi^{'}_{\rho} \bigl( (x^k,\lam^k); (p^k,q^k) \bigr) \le - \Phi_\rho (x^k,\lam^k) < 0,
\end{equation*}
which combined with \eqref{Phi'-one-side} yields that $\eta \ge 1$. This contradicts the choice of $\eta \in (0,1)$.
\end{proof}

\subsection{Solvability of the LC subproblem}\label{subsec:subproblem-solvability}

We next address the solvability of the LC subproblem. The $(\alpha,\beta)$-monotonicity condition plays a role analogous to the positive semidefiniteness of the quadratic term in SQP methods and guarantees that the LC subproblem admits a solution.

\begin{theorem}\label{thm:subproblem-solvability-1}
		Suppose that the following conditions hold:
		\begin{enumerate}
			\item The KKT system~\eqref{KKT} satisfies the $(\alpha,\beta)$-monotonicity condition at $(x,\lambda)$ for some $\alpha,\beta>0$;
			\item The linearized constraints are feasible, i.e., there exists a vector $d\in \R^n$ such that
			\begin{equation*}
				g^\nu_i (x) + \J_x g^\nu_i (x)d \le 0 \quad \text{for all}\ i\in[m_\nu], \nu\in [N].
			\end{equation*} 
		\end{enumerate}
		Then the LC subproblem~\eqref{KKT-linearized} at $(x,\lam)$ admits a solution.
	\end{theorem}
	
	\begin{proof}
	Since $\J_x F(x,\lam)$ is invertible, the LC subproblem~\eqref{KKT-linearized} is equivalent to the following system:
	\begin{equation}\label{mLCP}
	    \begin{array}{c}
	          0\le -G(x) + \J_x G(x) \J_x F(x,\lam)^{-1} (E(x) q +F(x,\lam)) \perp \lam + q \ge 0,\\
	          p = -\J_x F(x,\lam)^{-1}\left(E(x)q+F(x,\lam)\right).
	    \end{array}
	\end{equation}
	Let $\tlam:= \lam + q$. The LCP \eqref{mLCP} is equivalent to  
	\begin{equation}\label{LCP:thm-solvability}
	\begin{aligned}
	    &0\le -G(x) - \J_x G(x) \J_x F(x,\lam)^{-1}\left(E(x)\lam- F(x,\lam)\right)\\
	    &\qquad\qquad\qquad\qquad\qquad \qquad\qquad + \J_x G(x) \J_x F(x,\lam)^{-1}E(x)\tlam \perp \tlam \ge 0.
	\end{aligned}
	\end{equation}
	By the $(\alpha,\beta)$-monotonicity condition, the matrix $\J_x G(x) \J_x F(x,\lam)^{-1}E(x)$ is positive semidefinite. For any $u\in \R^m_+$ satisfying 
	\begin{equation*}
	\left\{\begin{array}{ll}
	     \J_x G(x) \J_x F(x,\lam)^{-1}E(x)u\ge 0,\\
	     u^\top\J_x G(x) \J_x F(x,\lam)^{-1}E(x)u=0,
	\end{array}\right.
    \end{equation*}
	we have
	\begin{equation*}
	    0 = u^\top\J_x G(x) \J_x F(x,\lam)^{-1}E(x)u \ge \beta \Vert \J_x G(x)^\top u \Vert^2. 
	\end{equation*}
	Consequently, $  \J_x G(x)^\top u  = 0$ and
	\begin{equation*}
	    \begin{aligned}
	        & u^\top \left( -G(x) - \J_x G(x) \J_x F(x,\lam)^{-1}\left(E(x)\lam-F(x,\lam)\right) \right)
	        =  -u^\top G(x).
	    \end{aligned}
	\end{equation*}
	By condition $(ii)$, there exists a vector $d\in \R^n$ such that 
	\begin{equation*}
	    G(x) + \J_x G(x)d \le 0.
	\end{equation*}
	Using $u \ge 0$, we obtain
	\begin{equation*}
	    -u^\top G(x) = u^\top(-G(x) - \J_x G(x)d) \ge 0.
	\end{equation*}
	Therefore, the LCP \eqref{LCP:thm-solvability} admits a solution $\tlam^*$  by Lemma~\ref{lemma:LCP-solvability}. Let 
	\begin{equation*}
	    \begin{aligned}
	           q^* &= \tlam^* - \lam, \\
	           p^* &= -\J_x F(x,\lam)^{-1}\left(E(x)q^*+F(x,\lam)\right).
	    \end{aligned}
	\end{equation*}
	The pair $(p^*,q^*)$ solves the LC subproblem~\eqref{KKT-linearized}.
	\end{proof}

% \begin{remark}
% Condition $(ii)$ in Theorem~\ref{thm:subproblem-solvability-1} merely requires the feasibility of the LC subproblem, which is a rather mild condition in the analysis.
% \end{remark}

It should be noted that some GNEPs may not satisfy the $(\alpha,\beta)$-monotonicity condition.
In such cases, the solvability of the LC subproblems cannot be guaranteed by Theorem~\ref{thm:subproblem-solvability-1}.
Nevertheless, these GNEPs are often player-convex, and can be reformulated as penalized NEPs sharing the same equilibria. This penalization technique has been widely used in \cite{ba2022exact}. Moreover, existence results for player-convex NEPs are well established and typically follow from fixed-point theorems such as Kakutani's fixed point theorem. Interested readers are referred to \cite[Section~11]{palomar2010convex} for a comprehensive discussion. 

In the following analysis, we extend the results of \cite{ba2022exact} by considering a more general setting where the private constraint sets of individual players may be unbounded.
To address this issue, we introduce restricted versions of these constraints so that each player's private feasible region becomes bounded, and show that the modified formulation is equivalent to the original one in most practical cases.
Based on this reformulation and the penalization framework, we further establish the solvability of the LC subproblems for GNEPs that do not satisfy the $(\alpha,\beta)$-monotonicity condition.
% Finally, we analyze the solvability of the LC subproblems arising in the internet switching model introduced at the beginning of this paper.

% \begin{theorem}\label{thm-existence-fixed-point}
%  Suppose that at the point $x\in \R^n$, for each $\nu \in [N]$, the following hold
%  \begin{enumerate}
%     \item[i)] The Hessian $\nabla^2_{x^\nu x^\nu} L^\nu (x^\nu,x^{-\nu},\lambda^\nu)$ is positive semi-definite;
%      \item[ii)] The constraint $g^\nu (x^\nu, x^{-\nu}) = \hg^\nu (x^\nu)$, i.e., the GNEP reduces to an NEP;
%      \item[iii)] The set $C^\nu := \left\{ d^\nu \mid \hg(x^\nu) + \J_{x^\nu} \hg^\nu(x^\nu) d^\nu \le 0  \right\}$ is non empty, compact and convex.
%  \end{enumerate}
%  Then LC subproblem~\eqref{subproblem} at the given point $x$ has at least a Nash equilibrium.
% \end{theorem}
% This is a corollary from Kakutani’s fixed-point theorem. We refer the reader to \cite[Section 11]{palomar2010convex} for more details.

For a given point $x$ and any $\nu \in [N]$, we partition the index set $[m_\nu]$ into two disjoint subsets $I^{\nu}_{\pr}$ and $I^{\nu}_{\np}$ such that $[m_\nu] = I^{\nu}_{\pr} \cup I^{\nu}_{\np}$. The set $I^{\nu}_{\pr}$ corresponds to the constraints that depend only on the private strategy $x^\nu$, while $I^{\nu}_{\np}$ collects those that are affected by other players' strategies. For simplicity, we write $g^\nu_i (x^\nu,x^{-\nu})$ by $g^\nu_i(x^\nu)$ for all $i\in I^{\nu}_{\pr}$ and $\nu\in[N]$.
The private strategy set $C^\nu \subset \R^{n_\nu}$ of the GNE subproblem~\eqref{subproblem} at $x$ is defined as
\begin{equation}
    C^\nu := \left\{ p^\nu \mid g^\nu_i (x^\nu) + \J_{x^\nu} g^\nu_i (x^\nu) p^\nu \le 0\ \text{for all}\ i\in I^{\nu}_{\pr}  \right\}.
\end{equation}
The set of all feasible strategies for player $\nu$ is defined by
\begin{equation}
     D^\nu:= \left\{ p\in \R^n    \mid g^\nu_i (x) + \J_x g^\nu_i (x) p \le 0\ \text{for all}\ i\in I^{\nu}_{\np} \right\}.
 \end{equation}
Accordingly, let $C^{-\nu}:= \Proj_{\mu \in [N]\setminus\{\nu\}} C^{\mu}$. Define $\hC^\nu$ by the restricted private strategy set
\begin{equation}\label{def-hC}
    \hC^\nu:= 
     C^\nu \bigcap \left\{ p^\nu \in \R^{n_\nu} \middle\vert  
    \begin{array}{ll}
          \exists\, p^{-\nu} \in C^{-\nu}\ \text{such that}\ p= (p^\nu,p^{-\nu}), {\rm and} \\
          g^\nu_i (x) + \J_x g^\nu_i (x) p \le 0\ \text{for all}\ i\in I^{\nu}_{\np}
    \end{array}   
     \right\}.
\end{equation}

\begin{lemma}
    The restricted private strategy set $\hC^\nu$ is a polyhedron for any $\nu \in [N]$. 
\end{lemma}

\begin{proof}
Let $C = \prod_{\nu\in[N]} C^\nu$, and let $\Proj_{\R^{n_\nu}}: \R^n \rightarrow \R^{n_\nu}$ denote the canonical projection satisfying $\Proj_{\R^{n_\nu}}(p) = p^\nu$ for any $p\in \R^n$. The restricted private strategy set $\hC^\nu$ can be rewritten as
 $$
    \hC^\nu = \Proj_{\R^{n_\nu}}( D^\nu \cap C ) \cap C^\nu.
 $$
 Since both $D^\nu$ and $C$ are polyhedra, their intersection $D^\nu \cap C$ is a polyhedron. Moreover, the image of a polyhedron under a linear operator is also a polyhedron (see \cite[Theorem 19.3]{rockafellar1970convex}). Hence, $\hC^\nu$ is a polyhedron, which completes the proof.
\end{proof}

This result is important in practice. In GNEPs, each player's private strategy set is commonly unbounded, preventing the direct application of classical fixed point theorems that rely on compactness. However, the non-private constraints provide an advantage, which ensures that the restricted strategy set $\hC^\nu$ is actually compact. For example, $\hC^\nu$ is compact at any $x\in \R^n_{++}$ for the GNE subproblem~\eqref{subproblem}
in Example~\ref{exm-internet}.

We now consider the restricted GNE subproblem at a given point $(x,\lam)$, where each player $\nu$ solves the quadratic programming problem 
\begin{equation}\label{subproblem-restricted}
     \begin{aligned}
         \min\limits_{p^\nu}\ \ & \grad_{x^\nu} \theta^\nu (x)^\top p^\nu + \frac{1}{2}(p^\nu)^\top \grad_{x^\nu x^\nu} L^\nu(x,\lam^\nu)p^\nu + (p^\nu)^\top \grad_{x^\nu x^{-\nu}}L^\nu (x,\lam^\nu)p^{-\nu}  \\
         {\rm s. t.}\ \ \ \ & 
         p^\nu \in \hC^\nu,\\
         & g_i^\nu (x) + \J_{x} g^\nu_i (x) p \le 0 \quad {\rm for}\ {\rm all}\ i\in I^{\nu}_{\np}.\\
    \end{aligned}
\end{equation}
It is worth noting that the restricted GNE subproblem and the original GNE subproblem share the same Nash equilibria.

\begin{theorem}\label{thm-equivalence-restricted-subproblem}
 The GNE subproblem~\eqref{subproblem} and the restricted GNE subproblem~\eqref{subproblem-restricted} admit the same set of Nash equilibria at any $(x,\lambda)$.
\end{theorem}

\begin{proof}
Since $\hC^\nu \subset C^\nu$ for all $\nu \in [N]$, it suffices to show that every GNE $\bp$ of~\eqref{subproblem-restricted} is also a GNE of~\eqref{subproblem}. We proceed by contradiction. Assume that $\bp$ is not a GNE of~\eqref{subproblem}. Then there exist $\hnu \in [N]$ and $p^{\hnu} \in C^{\hnu}$ such that
\begin{equation}\label{thm-restricted-1}
    g_i^{\hnu} (x) + \J_{x^{\hnu}} g^{\hnu}_i (x) p^{\hnu} + \J_{x^{-\hnu}} g^{\hnu}_i (x) \bp^{-\hnu} \le 0 \quad {\rm for}\ {\rm all}\ i\in I^{\hnu}_{\np},
\end{equation}
and
$$
\begin{aligned}
     & \grad_{x^{\hnu}} \theta^{\hnu}(x)^\top p^{\hnu} + \frac{1}{2}(p^{\hnu})^\top \grad_{x^{\hnu} x^{\hnu}} L^{\hnu}(x,\lam^{\hnu})p^{\hnu} + (p^{\hnu})^\top \grad_{x^{\hnu} x^{-{\hnu}}}L^{\hnu} (x,\lam^{\hnu})\bp^{-{\hnu}} \\
      < & \grad_{x^{\hnu}} \theta^{\hnu} (x)^\top \bp^{\hnu} + \frac{1}{2}(\bp^{\hnu})^\top \grad_{x^{\hnu} x^{\hnu}} L^{\hnu}(x,\lam^{\hnu})\bp^{\hnu} + (\bp^{\hnu})^\top \grad_{x^{\hnu} x^{-{\hnu}}}L^{\hnu} (x,\lam^{\hnu})\bp^{-{\hnu}}.
\end{aligned}
$$
However, since $\bp^{-\hnu} \in \hC^{-\hnu} \subset C^{-\hnu}$, it follows from \eqref{def-hC} and \eqref{thm-restricted-1} that $p^{\hnu} \in \hC^{\hnu}$. This contradicts the fact that $\bp^\nu$ is a global minimizer of the restricted subproblem~\eqref{subproblem-restricted}.
\end{proof}

Using the equivalence between the GNE subproblem and the restricted one, we now establish the existence of a solution to the LC subproblem.
We first consider the case when the GNEP is jointly-convex.
\begin{theorem}\label{thm-subproblem-solvability-jointly}
		Suppose that the following conditions hold:
		\begin{enumerate}
			\item The set $\hC^\nu$ is nonempty and compact for all $\nu \in [N]$;
			\item The GNEP \eqref{GNEP} is jointly-convex.
		\end{enumerate}
		Then the LC subproblem~\eqref{KKT-linearized} at $(x,\lambda)$ admits at least one solution.
\end{theorem}
\begin{proof}
By \cite[Corollary 8]{ba2022exact}, the restricted GNE subproblem~\eqref{subproblem-restricted} admits at least one GNE. Combining this with Theorem~\ref{thm-equivalence-restricted-subproblem}, we conclude that the GNE subproblem~\eqref{subproblem} admits at least one GNE. By condition $(ii)$, any such GNE yields a solution to the LC subproblem~\eqref{KKT-linearized}.
\end{proof}

If the joint convexity fails to hold, establishing the existence becomes more delicate.

\begin{theorem}\label{thm-subproblem-general-solvability}
 Suppose that the following conditions hold:
		\begin{enumerate}
			\item The set $\hC^\nu$ is nonempty and compact for all $\nu \in [N]$;
			\item The GNEP \eqref{GNEP} is player-convex;
			\item For any $p\in \hC = \Proj_{\nu \in [N]} \hC^\nu$ such that
			$$
			g^{\hnu}_{i} (x) + \J_{x} g_{i}^{\hnu} (x) p >  0 
			$$  
			for some $i\in I^{\hnu}_{\np}$ and $\hnu\in[N]$, there exists a vector $\hat{p}^{\hnu}\in \hC^{\hnu}$ such that
			\begin{equation*}
			g^{\hnu}_{j} (x) + \J_{x^{\hnu}} g_{j}^{\hnu} (x) {\hat{p}}^{\hnu} + \J_{x^{-{\hnu}}} g^{\hnu}_j (x) p^{-\hnu} \le  0\quad \text{for all}\ j\in I^{\hnu}_{\np}. 
			\end{equation*}
		\end{enumerate}
		Then the LC subproblem~\eqref{KKT-linearized} at $(x,\lambda)$ admits at least one solution.
\end{theorem}

\begin{proof}
We provide a concise proof using the techniques from \cite{ba2022exact}. Consider the following NEP with a penalty parameter $\varrho>0$, where each player solves 
\begin{equation}\label{Subproblem-NEP}
        \begin{aligned}
            \min\limits_{p^\nu \in \R^{n_{\nu}}}\ \ & \frac{1}{2} (p^\nu)^\top 
            \nabla^2_{x^\nu x^\nu} L^\nu (x,\lam^\nu) p^\nu + (p^\nu)^\top  \nabla^2_{x^\nu x^{-\nu}} L^\nu (x,\lam^\nu) p^{-\nu}\\
           & \qquad \qquad + \grad_{x^\nu} \theta^\nu (x)^\top p^\nu + \varrho \sum\limits_{j\in I^{\nu}_{\np}} \left( g_j^\nu (x) + \J_{x} g^\nu_j (x) p\right)_+ \\
            {\rm s. t.}\ \ \ \ 
            & p^\nu \in \hC^\nu.
        \end{aligned}
    \end{equation}
Let $\bp$ be a Nash equilibrium of \eqref{Subproblem-NEP}, whose existence is guaranteed by \cite[Section~11]{palomar2010convex}. We show by contradiction that $\bp$ is also a GNE of \eqref{subproblem-restricted} when $\varrho$ is sufficiently large. Assume that $\bp$ is not a GNE of \eqref{subproblem-restricted} for any $\varrho>0$. There exists $\hnu\in [N]$ such that $\bp \notin D^{\hnu}$. Let 
\begin{equation*}
    p^{\hnu} = \arg\min_{v\in \hC^{\hnu}} \left\{ \Vert v - \bp^{\hnu} \Vert \, \middle\vert\,  g_i^{\hnu} (x) + \J_{x^{\hnu}} g^{\hnu}_j (x) v +\J_{x^{-\hnu}} g^{\hnu}_j (x) \bp^{-\hnu} \le 0,\ \forall j\in I^{\hnu}_{\np}  \right\}.
\end{equation*}
Since $\bp$ is a Nash equilibrium of \eqref{Subproblem-NEP}, we have 
\begin{equation}\label{thm-existence-2}
    \begin{aligned}
     & \grad_{x^{\hnu}} \theta^{\hnu} (x)^\top p^{\hnu} + \frac{1}{2}(p^{\hnu})^\top \grad_{x^{\hnu} x^{\hnu}} L^{\hnu}(x,\lam^{\hnu})p^{\hnu} + (p^{\hnu})^\top \grad_{x^{\hnu} x^{-{\hnu}}}L^{\hnu} (x,\lam^{\hnu})\bp^{-{\hnu}} \\
     \ge & \grad_{x^{\hnu}} \theta^{\hnu} (x)^\top \bp^{\hnu} + \frac{1}{2}(\bp^{\hnu})^\top \grad_{x^{\hnu} x^{\hnu}} L^{\hnu}(x,\lam^{\hnu})\bp^{\hnu} + (\bp^{\hnu})^\top \grad_{x^{\hnu} x^{-{\hnu}}}L^{\hnu} (x,\lam^{\hnu})\bp^{-{\hnu}}\\
     & \qquad \qquad + \varrho \sum\limits_{j\in I^{\hnu}_{\np}} \left( g_j^{\hnu} (x) + \J_{x} g^{\hnu}_j (x) \bp\right)_+.
\end{aligned}
\end{equation}    
Define the positive constant 
\begin{equation*}
     C_1:= \Vert \grad_{x^{\hnu}} \theta^{\hnu} (x) + \grad_{x^{\hnu} x^{-{\hnu}}}L^{\hnu} (x,\lam^{\hnu})\bp^{-{\hnu}} \Vert + \sup\limits_{ v\in \hC^{\hnu}} \Vert \grad_{x^{\hnu} x^{\hnu}} L^{\hnu}(x,\lam^{\hnu}) v \Vert.\\
\end{equation*}
By Hoffman's lemma \cite{hoffman2003approximate}, there exists a constant $C_2>0$ such that 
\begin{equation*}
    0<\Vert p^{\hnu} - \bp^{\hnu} \Vert \le C_2 \sum\limits_{j\in I^{\hnu}_{\np}} \left( g_j^{\hnu} (x) + \J_{x} g^{\hnu}_j (x) \bp\right)_+.
\end{equation*}
Choose $\varrho> C_1 C_2$. Subtracting the right-hand side of \eqref{thm-existence-2} from its left-hand side yields
\begin{equation*}
   0 \ge \left( - C_1 + \frac{\varrho}{C_2}\right)\Vert p^{\hnu} - \bp^{\hnu} \Vert > 0,
\end{equation*}
which contradicts the choice of $\varrho$. Thus, $\bp$ is also a GNE of \eqref{subproblem-restricted}. By Theorem~\ref{thm-equivalence-restricted-subproblem} and condition $(ii)$, $\bp$ yields a solution to the LC subproblem~\eqref{KKT-linearized}. 
\end{proof}

To illustrate the above theorems on solvability results, we examine the subproblem arising from Example~\ref{exm-internet}. The analysis is given in Appendix~\ref{appendix:further analysis for the internet switching model}; see Proposition~\ref{exm-subproblem-solvability}.

\section{Global convergence}\label{sec:global-convergence}
In this section, we establish the global convergence of the SLCP method and provide sufficient conditions for the boundedness of the iterates. The analysis is based on the proposed merit function and the $(\alpha,\beta)$-monotonicity condition.

\begin{theorem}\label{thm:global convergence}
Suppose that the sequence $\{(x^k,\lambda^k)\}$ generated by Algorithm~\ref{algslcp} satisfies the following conditions at every iterate:
		\begin{enumerate}
			\item The KKT system~\eqref{KKT} satisfies the $(\alpha,\beta)$-monotonicity condition at $(x^k,\lambda^k)$ for some $\alpha,\beta>0$;
			\item There exists a vector $d\in \R^n$ such that
			\begin{equation}\label{slater}
				g^\nu_i (x^k) + \J_x g^\nu_i (x^k)d \le 0 \quad \text{for all}\ i\in[m_\nu], \nu\in [N].
			\end{equation} 
		\end{enumerate}
		Then, for any $\rho \ge \frac{2\alpha^2}{\beta} + 1$, the sequence $\{(x^k,\lambda^k)\}$ either terminates at a KKT pair of \eqref{KKT}, or any accumulation point $(\bar{x},\bar{\lambda})$ satisfying
		\begin{equation}\label{strict slater}
		\left\{ d \mid g^\nu_i (\bx) + \J_x g^\nu_i (\bx)d < 0 \quad \text{for all}\ i\in[m_\nu], \nu\in [N] \right\} \neq \emptyset
		\end{equation}
		is a KKT pair of \eqref{GNEP}.
	Moreover, if $g^\nu_i (x)$ is convex for all $i\in [m_\nu], \nu\in[N]$, conditions \eqref{slater} and \eqref{strict slater} can be replaced by the Slater condition for $G(x)$.
\end{theorem}

\begin{proof}
We first show that if $(x^k,\lam^k)$ converges to a point $(\bx,\blam)$ satisfying \eqref{strict slater}, then the solutions $(p^k,q^k)$ of the LC subproblem~\eqref{KKT-linearized} at $(x^k,\lam^k)$ are bounded. Such a solution sequence exists for all $(x^k,\lam^k)$ due to Theorem~\ref{thm:subproblem-solvability-1}. Define 
\begin{equation*}
\begin{aligned}
    h^k& :=  -G(x^k) - \J_x G(x^k) \J_x F(x^k,\lam^k)^{-1}\left(E(x^k)\lam^k- F(x^k,\lam^k)\right),\\
    M^k&:= \J_x G(x^k) \J_x F(x^k,\lam^k)^{-1}E(x^k).
\end{aligned}
\end{equation*}
% \begin{equation}\label{LCP:thm-global convergence}
% 	\begin{aligned}
% 	    &0\le -G(x^k) - \J_x G(x^k) \J_x F(x^k,\lam^k)^{-1}\left(E(x^k)\lam^k- F(x^k,\lam^k)\right)\\
% 	    &\qquad\qquad\qquad\qquad\qquad \qquad\qquad + \J_x G(x^k) \J_x F(x^k,\lam^k)^{-1}E(x^k)\lam^k \perp \tlam \ge 0,
% 	\end{aligned}
% 	\end{equation}
Since $(p^k,q^k)$ solves the LC subproblem~\eqref{KKT-linearized} at $(x^k,\lam^k)$, we have 
\begin{equation}\label{equivalence-MLCP-LCP}
    \begin{aligned}
        &\tlam^k:= \lam^k + q^k \in {\rm SOL}(h^k,M^k),\\ 
        & p^k = -\J_x F(x^k,\lam^k)^{-1}\left(E(x^k)\tlam^k - E(x^k) \lam^k +F(x^k,\lam^k)\right).
    \end{aligned}
\end{equation}
By the $(\alpha,\beta)$-monotonicity condition, 
$$
\sup_{k}\Vert \J_x F(x^k,\lam^k)^{-1} \Vert \le \alpha,
$$ 
which together with $(x^k,\lam^k) \to (\bx,\blam)$ implies that
\begin{equation*}
    \lim\limits_{k\to \infty} h^k = \bar{h}, \quad \lim\limits_{k\to \infty} M^k = \bar{M},
\end{equation*}
where
\begin{equation*}
    \begin{aligned}
           \bar{h} &= -G(\bx) - \J_x G(\bx) \J_x F(\bx,\blam)^{-1}\left(E(\bx)\blam- F(\bx,\blam)\right),\\
           \bar{M} &= \J_x G(\bx) \J_x F(\bx,\blam)^{-1}E(\bx).
    \end{aligned}
\end{equation*}
To apply Lemma~\ref{lemma:LCP stability}, consider nonzero $u\ge 0$ with $\bar{M}u\ge 0$ and $u^\top \bar{M} u =0$. The $(\alpha,\beta)$-monotonicity condition gives
\begin{equation*}
    0 = u^\top \bar{M}u \ge \beta \Vert \J_x G(\bx)^\top u \Vert^2, 
\end{equation*}
and implies that $\J_x G(\bx)^\top u = 0$. Thus, 
\begin{equation}\label{ineq-3}
    \bar{h}^\top u = - G(\bx)^\top u.
\end{equation}
By \eqref{strict slater}, there exists a vector $d\in \R^n$ such that 
\begin{equation}\label{JGd<0}
    G(\bx) + \J_x G(\bx) d < 0.
\end{equation}
Multiplying \eqref{JGd<0} by $u \in \R^m_+$ yields
\begin{equation*}
    0>u^\top \left( G(\bx) + \J_x G(\bx) d \right) = u^\top G(\bx).
\end{equation*}
which combined with \eqref{ineq-3} verifies
\begin{equation*}
		\bar{M}u\ge 0,\ u^\top \bar{M} u = 0 \implies u^\top \bar{h} > 0.
\end{equation*}
Thus, ${\rm SOL}(h^k,M^k)$ is uniformly bounded by Lemma~\ref{lemma:LCP stability}. This together with~\eqref{equivalence-MLCP-LCP} implies the boundedness of the sequence $\{(p^k,q^k)\}$. 

Next, suppose that $(\bx,\blam)$ is the limit point along an infinite set  
$K_0 \subset \mathbb{N}$, i.e.,
$$
(x^k,\lam^k) \to (\bx,\blam) \quad \text{as } k\in K_0,\ k\to \infty.
$$
Then the sequence $\{(p^k,q^k)\}_{k\in K_0}$ is bounded. Let $(\bp,\bq)$ be an accumulation point of $\{(p^k,q^k)\}_{k\in K_0}$. By continuity, $(\bp,\bq)$ solves the LC subproblem~\eqref{KKT-linearized} at $(\bx,\blam)$. We show by contradiction that $(\bx,\blam)$ is a KKT pair.
Assume the contrary. It follows from Proposition~\ref{prop:Phi = 0 implies KKT} that 
$\Phi_\rho (\bx,\blam) >0$. By the descent property of the SLCP method,
\begin{equation*}
\Phi_\rho (\bx,\blam) \le \prod\limits_{k \in K_0} (1- \eta \tau_k) \Phi_\rho (x^k,\lam^k)
\le \Phi_\rho (x^0,\lam^0) \left( \prod\limits_{k\in K_0} (1- \eta \tau_k) \right),
\end{equation*}
which implies
\begin{equation}\label{prf:tau_k->0}
    \sum\limits_{k\in K_0} \tau_k < + \infty,\quad \text{and thus }     \lim\limits_{\overset{k \in K_0}{k\rightarrow + \infty} } \tau_k = 0.
\end{equation}
Since $(\bp,\bq)$ solves the LC subproblem~\eqref{KKT-linearized} at $(\bx,\blam)$, Lemma~\ref{lemma:linesearch terminate} ensures the existence of the step length $\bar{\tau} = \frac{\tau_0}{2^{j}}>0$ for some $j$ such that 
\begin{equation}\label{prf:Phi<= (1-eta/2)}
    \Phi_\rho (\bx + \bar{\tau}\bp, \blam + \bar{\tau}\bq) \le (1- 2 \eta \bar{\tau})\Phi_\rho (\bx,\blam).
\end{equation}
However, by \eqref{prf:tau_k->0} and the line search procedure, there exists a sufficiently large $\underline{k}$ such that for any $k\in K_0$, $k \ge \underline{k}$,
% \begin{equation}
%     \Phi_\rho (x^k+\bar{\tau}p^k, \lam^k + \bar{\tau}q^k) >  (1-\eta \bar{\tau})\Phi_\rho (x^k,\lam^k) + \eta \bar{\tau} \Phi^{'}_{\rho} ( (x^k,\lam^k); (p^k,q^k) ).
% \end{equation}
\begin{equation*}
    \Phi_\rho (x^k+\bar{\tau}p^k, \lam^k + \bar{\tau}q^k) >  (1-\eta \bar{\tau})\Phi_\rho (x^k,\lam^k).
\end{equation*}
Letting $k\in K_0$ and $k\rightarrow + \infty$ yields
\begin{equation*}
    \Phi_\rho (\bx + \bar{\tau}\bp, \blam + \bar{\tau}\bq) \ge (1-\eta \bar{\tau}) \Phi_\rho (\bx,\blam),
\end{equation*}
which contradicts the inequality \eqref{prf:Phi<= (1-eta/2)}. Thus, $(\bx,\blam)$ is a KKT pair. 

Finally, suppose that $G(\cdot)$ is convex and the Slater condition holds. Then there exists a point $\hat{x}$ such that $G(\hat{x}) < 0$. For any $x\in \R^n$, let $d = \hat{x} -x$, and we have
$$
        g^\nu_i (x) + \J_x g^\nu_i (x)d \le g^\nu_i (\hat{x}) <0\quad \text{for all } i\in [m_\nu], \nu \in [N],
$$
which ensures that both \eqref{slater} and \eqref{strict slater} hold.
\end{proof}

The boundedness of the iterates can be further guaranteed by a mild additional assumption.

\begin{theorem}\label{thm:boundedness}
 Let $\{ (x^k,\lam^k)\}$ be the sequence generated by Algorithm~\ref{algslcp}. Suppose that the following conditions hold:
 \begin{enumerate}
     \item The KKT system~\eqref{KKT} satisfies the $(\alpha,\beta)$-monotonicity condition at all $(x^k,\lambda^k)$ for some $\alpha,\beta>0$;
     \item $\lim\limits_{\Vert x \Vert \rightarrow +\infty} \Vert \left(G(x) \right)_+ \Vert = + \infty$;
    %  \item The Extended Mangasarian-Fromovitz Constraint Qualification (EMFCQ) holds for each player, i.e., for all $\nu\in[N]$ and for all $x\in \R^n$,
    %  \begin{equation}
    %      \exists d^\nu \in \R^{n_\nu}\ {\rm s.t.} \ \J_{x^\nu} g^\nu_i (x) d^\nu < 0\ \forall i\in J^\nu_+ \cup {J}^{\nu}_{0}.
    %  \end{equation}
    \item For any $x\in \R^n$, there exists a direction $d \in \R^n$ such that
    \begin{equation*}
         g^\nu_i(x) + \J_x g^\nu_i (x) d < 0 \quad \text{for all}\ i\in [m_\nu], \nu \in [N].
    \end{equation*}
 \end{enumerate}
      Then, for any starting point $(x^0,\lam^0)\in \R^n \times \R^m_+$ and $\rho \ge \frac{2\alpha^2}{\beta} + 1$, the sequence $\{ (x^k, \lambda^k) \}$ remains bounded. Moreover, the sequence either terminates at a KKT pair, or any accumulation point 
	  is a KKT pair of \eqref{GNEP}.
\end{theorem}

\begin{proof}
We proceed by contradiction. Assume that $\Vert (x^k,\lam^k) \Vert$ is unbounded. Without loss of generality, let 
\begin{equation*}
    \lim\limits_{k\to \infty} \| (x^k,\lam^k) \| = +\infty.
\end{equation*}
We consider two cases. 

\textit{Case 1:} $\Vert x^k \Vert \rightarrow +\infty$.  Recall that 
\begin{equation*}
    \Phi_\rho(x^k,\lam^k) = \left( - (\lam^k)^\top G(x^k) \right)_+
    + \frac{\rho}{2} \Vert F(x^k,\lambda^k) \Vert^2 + \sum\limits_{\nu \in [N]} \sum\limits_{i\in[m_\nu]} \left( g^\nu_i ({x^k})\right)_+.
\end{equation*}
By condition $(ii)$, we have
$$
\Phi_\rho (x^k,\lam^k) \ge \sum\limits_{\nu \in [N]} \sum\limits_{i\in[m_\nu]} \left( g^\nu_i ({x}^k)\right)_+ \rightarrow +\infty,
$$
which contradicts the fact that 
$\Phi_\rho(x^k,\lambda^k)$ is nonincreasing along the iterations. 

\textit{Case 2:} $\{x^k\}$ is bounded but $\Vert \lambda^k \Vert \rightarrow +\infty$. 
Since $\Phi_\rho$ is nonincreasing and 
$$
        \Phi_\rho(x^k,\lam^k) \ge \frac{\rho}{2}\Vert F(x^k,\lam^k) \Vert^2,
$$
we get that $\Vert F(x^k,\lam^k) \Vert$ is bounded. Let 
$$
 u:= \lim_{k\rightarrow +\infty} \frac{\lam^k}{\Vert \lam^k \Vert}  \in \R^m_+ \setminus \{ 0\},
$$
and let $\lim_{k \rightarrow \infty} x^k = \bx$ (after taking a subsequence if necessary). Then
\begin{equation*}
    \lim\limits_{k \rightarrow +\infty } \frac{ F(x^k,\lam^k) }{\Vert \lam^k \Vert} = E(\bx) u = 0.
\end{equation*}
It follows from the $(\alpha,\beta)$-monotonicity condition that for any $\lam \in \R^m_+$, 
\begin{equation*}
    0 = u^\top \J_x G(\bx) \J_x F(\bx,\lam)^{-1} E(\bx) u \ge \beta \Vert \J_x G(\bx)^\top u \Vert^2.
\end{equation*}
Thus, $\J_x G(\bx)^\top u = 0$ and for any $d\in \R^n$,
\begin{equation*}
    u^\top G(\bx) = u^\top (G(\bx) + \J_x G(\bx)d).
\end{equation*}
By condition $(iii)$, there exists $d\in \R^n$ such that $G(\bx) + \J_x G(\bx) d < 0$. Since $ u \in \R^m_+ \setminus \{0\}$, we obtain that
\begin{equation*}
    u^\top G(\bx) = u^\top (G(\bx) + \J_x G(\bx)d) < 0.
\end{equation*}
Consequently, there exists $c_1<0$ such that for all sufficiently large $k$,
\begin{equation*}
    (\lam^k)^\top G(x^k) \le \frac{c_1}{2} \Vert \lam^k  \Vert \rightarrow - \infty,
\end{equation*}
which implies  
\begin{equation*}
    \Phi_\rho (x^k,\lam^k) \ge \left( -(\lam^k)^\top G(x^k) \right)_+ \rightarrow +\infty. 
\end{equation*}
This contradicts the boundedness of $\Phi_\rho (x^k,\lam^k)$. The remaining conclusion follows from Theorem~\ref{thm:global convergence}.
\end{proof}

Condition $(iii)$ holds when $G(x)$ is convex and the Slater condition holds. When the GNEP reduces to an NLP problem, the three conditions have a one-to-one correspondence with \cite[Theorem 3.3]{han1977globally}, which establishes the global convergence of the classical SQP method with exact Lagrangian Hessians.

Analogous to the Newton-type methods, the SLCP method also achieves local quadratic convergence when the solution satisfies appropriate regularity conditions. However, these regularity conditions are relatively more involved for GNEPs. In the following section, we will discuss these regularity conditions in detail.

\section{Local quadratic convergence}\label{sec:local convergence}
In this section, we present sufficient conditions for the local quadratic convergence of the SLCP method. These conditions are established based on the concepts of hemistability and semistability. We further show that these two properties are not equivalent for GNEPs, in contrast to the case of standard NLP problems. Moreover, an error bound is derived from semistability.

We adopt the stability definitions from \cite{bonnans1994local} for the local quadratic convergence of Algorithm~\ref{algslcp}. Recall that $\N_{\R^m_+} (\lam)$ denotes the normal cone to $\R^m_+$ at $\lam\in \R^m$, and 
\begin{equation*}
    0\le \lam \perp -G(x) \ge 0\ \Longleftrightarrow\ 0 \in -G(x) + \N_{\R^m_+} (\lam).
\end{equation*}

\begin{definition}[Semistable and hemistable points]
 Let $(\bar{x}, \bar{\lambda})$ be a KKT pair of the GNEP \eqref{GNEP}. 
 \begin{enumerate}
     \item The pair $(\bar{x}, \bar{\lambda})$ is said to be \emph{semistable} if there exist constants $c_1,c_2>0$ such that any solution $(x,\lam)$ to the perturbed KKT system
    \begin{equation}\label{def-semistable}
        \begin{pmatrix}
            a\\[-2pt]
            b
        \end{pmatrix}
        \in 
        \begin{pmatrix}
            F(x, \lambda)\\[-2pt]
            -G(x)
        \end{pmatrix}
        + 
        \N_{\R^n \times \R^m_+}(x, \lambda),
    \end{equation}
    satisfying $\Vert x - \bar{x} \Vert + \Vert \lambda - \bar{\lambda} \Vert \le c_1$, also satisfies
    \[
        \Vert x - \bar{x} \Vert + \Vert \lambda - \bar{\lambda} \Vert \le c_2 (\Vert a \Vert + \Vert b \Vert).
    \]
    \item The pair $(\bar{x}, \bar{\lambda})$ is said to be \emph{hemistable} if for all $\delta>0$, there exists $\epsilon>0$ such that, given $(\tx,\tlam)$ with 
    $$
    \Vert \tx-\bar{x} \Vert+ \Vert \tlam- \blam \Vert + \| \J_x F(\tx,\tlam) - \J_x F(\bx,\blam) \| + \| E(\tx)-E(\bx)\| \le \epsilon,
    $$
    the LC subproblem at $(\tx,\tlam)$
		\begin{equation*}
			\begin{array}{c}
				\J_x {F(\tx,\tlam)}(x-\tx)+ E(\tx)(\lambda - \tlam)+ {F(\tx,\tlam)} = 0,\\
				0 \le {\lambda} \perp -{G( \tx)} - \J_x {G( \tx)}(x-\tx) \ge 0
			\end{array}
		\end{equation*} 
		admits a solution $(x,\lambda)$ satisfying $\Vert x - \bar{x} \Vert + \Vert \lambda - \bar{\lambda} \Vert \le \delta$.
 \end{enumerate}
\end{definition}

\begin{remark}
The above definitions highlight two stability properties. Semistability guarantees the local calmness property of the solution mapping with respect to perturbations, while hemistability concerns the existence of approximate solutions to the perturbed linearized subproblem. Generally, neither semistability nor hemistability implies the other.
\end{remark}

As noted in \cite[Remark 2.4]{bonnans1994local}, both stability properties are guaranteed by the strong regularity of the system, a concept introduced by Robinson \cite{robinson1980strongly}. We now recall its definition. Let $\S(a,b): \R^n\times \R^m \rightrightarrows \R^n\times \R^m$ be the set-valued mapping that assigns each $(a,b)$ the set of solutions to the perturbed LC subproblem at $(\bx,\blam)$:
 \begin{equation}\label{def-perturbed-LC-subproblem}
			\begin{array}{c}
				\J_x {F(\bx,\blam)}(x-\bx)+ E(\bx)(\lambda - \blam)+ {F(\bx,\blam)} = a,\\
				0 \le {\lambda} \perp -b-{G( \bx)} - \J_x {G( \bx)}(x-\bx) \ge 0.
			\end{array}
		\end{equation}

\begin{definition}[Strong regularity]\label{def-strong regularity}
 Let $(\bx,\blam)$ be a KKT pair of \eqref{GNEP}. The KKT system~\eqref{KKT} is said to be \emph{strongly regular} at $(\bx,\blam)$ if there exist open neighborhoods $U$ of $(0,0)$ and $V$ of $(\bx,\blam)$ such that the mapping $(a,b) \mapsto \S(a,b) \cap V$ is a single-valued Lipschitz continuous function from $U$ to $V$.
\end{definition}

% \begin{remark}\label{NLP-semistable-implies-hemistable}
% Generally, none of the two properties of semistability and hemistability is implied
% by the other. However, for an NLP problem, if $(\bx,\blam)$ is semistable and $\bx$ is a local solution, then $(\bx,\blam)$ is hemistable \cite{bonnans1994local}. Both these two properties can be derived by the strong regularity \cite{robinson1980strongly}.
% \end{remark}

% Following from \cite[Theorem 2.1 and 2.3]{bonnans1994local}, if the SLCP method generates iterates converging to a semistable point, then the convergence is quadratic. Moreover, the combination of semistability and hemistability guarantee the local quadratic convergence of the SLCP method, as stated in the following theorem.
Following \cite[Theorem 2.1 and 2.3]{bonnans1994local}, the local convergence of the SLCP method is ensured by semistability and hemistability, as stated in the following theorem.

\begin{theorem}\label{thm:local quadratic convergence}
Suppose that $(\bar{x},\bar{\lambda})$ is a semistable solution of \eqref{KKT}. The following statements hold:
\begin{enumerate}
    \item If the sequence $\{(x^k,\lam^k)\}$ generated by Algorithm~\ref{algslcp} taking full steps (i.e., step length $\tau_k = 1$) for sufficiently large $k$ converges to $(\bx,\blam)$, then the local convergence rate is quadratic;
    \item If $(\bx,\blam)$ is also hemistable, then there exists $\epsilon>0$ such that, whenever $\Vert x^0 - \bar{x} \Vert + \Vert \lambda^0 - \bar{\lambda} \Vert < \epsilon$, the sequence $\{ (x^k,\lam^k) \}$ generated by Algorithm~\ref{algslcp} taking full steps converges quadratically to $(\bar{x},\bar{\lambda})$.
\end{enumerate}

\end{theorem}

\begin{proof}
     If the full step is taken at each iteration of Algorithm~\ref{algslcp}, then the update of $(x^k,\lam^k)$ can be reformulated as 
    \begin{equation*}
    \begin{aligned}
      & \begin{pmatrix}
          0\\
          0
      \end{pmatrix}
      \in 
      \begin{pmatrix}
          F(x^k,\lam^k)\\
          -G(x^k)
      \end{pmatrix}
      +\begin{pmatrix}
          \J_x F(x^k,\lam^k) & E(x^k)\\
          -\J_x G(x^k) & 0
      \end{pmatrix}
      \begin{pmatrix}
          x^{k+1}-x^{k}\\
          \lam^{k+1} - \lam^{k}
      \end{pmatrix}\\
      &\qquad \qquad \qquad \qquad \qquad \qquad \qquad \qquad \qquad \qquad + \N_{\R^{n} \times \R^m_{+}}(x^{k+1},\lam^{k+1}).
    \end{aligned}
    \end{equation*}
    The results follow directly from \cite[Theorem 2.1 and 2.3]{bonnans1994local}.
\end{proof}

Define the index sets associated with $(\bx,\blam)$ as
\begin{equation}
\begin{aligned}
I^\nu_+ := \left\{ i\in [m_\nu] \mid  \blam^\nu_i > 0 = g^\nu_i (\bx)  \right\},\\
    I^\nu_0 := \left\{ i\in [m_\nu] \mid \blam^\nu_i = 0 =g^\nu_i (\bx)  \right\},\\
    I^\nu_- := \left\{ i\in [m_\nu] \mid \blam^\nu_i = 0> g^\nu_i (\bx)  \right\}.
\end{aligned}
\end{equation}
The following theorem provides an exact characterization of semistable solutions.

\begin{theorem}\label{semistable}
 A KKT pair $(\bx,\blam)$ of \eqref{GNEP} is semistable if and only if the following mixed LCP admits $(\dx,\dlam)=(0,0)$ as its unique solution:
 \begin{equation}\label{unique}
     \left\{\begin{array}{lll}
          \J_x F(\bx,\blam) \dx + E(\bx)\dlam = 0, &\\
          \J_x g^\nu_i (\bx)\dx = 0 &\qquad \text{for all}\ \nu\in[N], i\in I^\nu_+;\\
          0 \le \dlam^\nu_i \perp -\J_x g^\nu_i (\bx) \dx \ge 0 &\qquad \text{for all} \ \nu\in[N], i\in I^\nu_0;\\
          \dlam^\nu_i = 0 &\qquad  \text{for all} \ \nu\in[N], i\in I^\nu_-.
     \end{array}\right.
 \end{equation}
\end{theorem}

\begin{proof}
    By \cite[Corollary 2.3]{dontchev1997characterizations}, the KKT pair $(\bx,\blam)$ is semistable if and only if $(x,\lambda) = (\bx,\blam)$ is a locally unique solution to the mixed LCP:
    \begin{equation}\label{locally-unique}
    \begin{array}{c}
           F(\bx,\blam) + \J_x F(\bx,\blam) (x-\bx) + E(\bx)(\lam - \blam) = 0,\\
          0 \le \lam \perp -G(\bx) - \J_x G(\bx)(x - \bx) \ge 0.
    \end{array}
    \end{equation}
    We now show that this is equivalent to \eqref{unique} admitting $(\dx,\dlam) = (0,0)$ as a unique solution. 
    
    Assume that $ (\bx,\blam)$ is a locally unique solution of \eqref{locally-unique}. If $(0,0)$ is not a unique solution to \eqref{unique}, there exists a sequence of nontrivial solutions $\{(\dx^k,\dlam^k)\}$ of \eqref{unique} converging to $(0,0)$. The sequence $\{(\bx+\dx^k,\blam + \dlam^k)\}$ then solves \eqref{locally-unique} and converges to $(\bx,\blam)$, contradicting the local uniqueness of $(\bx,\blam)$. Conversely, assume that $(\dx,\dlam)=(0,0)$ solves \eqref{unique} uniquely.
    If $(\bx,\blam)$ is not a locally unique solution, then there exists a sequence of solutions $\{(x^k,\lam^k)\}$ converging to $(\bx,\blam)$. For sufficiently large $k$, the complementarity conditions imply that
    \begin{equation*}
        \left\{\begin{array}{lll}
               -g^\nu_i(\bx) - \J_x g^\nu_i(\bx) (x^k-\bx) = 0 &\quad \text{for all}\ \nu\in[N], i\in I^\nu_+;  \\
              0 \le \lam^{k,\nu}_i \perp -g^\nu_i (\bx) - \J_x g^\nu_i (\bx) (x^k-\bx) \ge 0 &\quad \text{for all}\ \nu\in[N], i\in I^\nu_0;\\
              \lam^{k,\nu}_i = 0 &\quad \text{for all}\ \nu\in[N], i\in I^\nu_-.
        \end{array}\right.
    \end{equation*}
    Let $(\dx^k,\dlam^k) = (x^k-\bx,\lam^k-\blam)$. Using 
    \begin{equation*}
        \left\{\begin{array}{lll}
              g^\nu_i(\bx) = 0 &\quad \text{for all}\ \nu\in[N], i\in I^\nu_+ \cup I^\nu_0;   \\
              \blam^\nu_i = 0 &\quad \text{for all}\ \nu\in[N], i\in I^\nu_0 \cup I^\nu_-,
        \end{array}\right.
    \end{equation*}
     it follows that $\{(\dx^k,\dlam^k)\}$ solves \eqref{unique} and converges to $(0,0)$, which contradicts the uniqueness of the trivial solution.
\end{proof}

\begin{remark}
When $\bx$ is a local minimizer of a nonlinear programming problem, the conditions in Theorem~\ref{semistable} reduce to the second-order sufficient condition (SOSC) and the strict Mangasarian--Fromovitz constraint qualification (SMFCQ).
\end{remark}

The exact characterization of semistable KKT pairs implies that the $(\alpha,\beta)$-monotonicity condition with the SMFCQ for the constraint system $G(x) \le 0$ is a sufficient condition for semistability, as stated by the following corollary.
Consequently, Theorem~\ref{thm:local quadratic convergence} ensures the local quadratic convergence of the SLCP method under these two conditions.

\begin{corollary}\label{coro:monotone-implies-semistable}
 Let $(\bx,\blam)$ be a KKT pair of \eqref{GNEP}. Suppose that the following conditions hold:
 \begin{enumerate}
     \item The KKT system~\eqref{KKT} satisfies the $(\alpha,\beta)$-monotonicity condition at $(\bx,\blam)$ for some $\alpha,\beta>0$;
     \item The strict Mangasarian--Fromovitz constraint qualification holds for the constraint system $G(x) \le 0$ at $(\bx,\blam)$. 
 \end{enumerate}
 Then the KKT pair $(\bx,\blam)$ is semistable.
\end{corollary}
\begin{proof}
 We proceed by contradiction. Assume that $(\bx,\blam)$ is not semistable. Then by Theorem~\ref{semistable}, there exists a nontrivial solution $(\dx,\dlam)$ of \eqref{unique}. This combined with the $(\alpha,\beta)$-monotonicity condition implies that 
 $$
         0 =(\dlam)^\top  \J_x G(\bx) \J_x F(\bx,\blam)^{-1} E(\bx) \dlam \ge \beta \| \J_x G(\bx)^\top \dlam \|^2 .
 $$
 It follows that $\J_x G(\bx)^\top \dlam = 0$.  Thus, there exists nonzero $\dlam \in \R^m$ such that $\dlam^\nu_i\ge 0$ for all $\nu \in [N]$, $i\in I^\nu_0$, and
 $$
        \sum\limits_{\nu\in [N]} \left(\sum\limits_{i\in I^\nu_+}  \dlam^\nu_i \nabla_x g^\nu_i(\bx) + \sum\limits_{i\in I^\nu_0} \dlam^\nu_i \nabla_x g^\nu_i(\bx) \right) =0.
 $$
 By the theorems of alternative~\cite{Mangasarian1969}, the SMFCQ fails for the constraint $G(x) \le 0$, which contradicts the assumption.
\end{proof}

Semistability also implies an error bound, which measures the distance between an iterate and a semistable solution $(\bx,\blam)$ when the KKT residual is sufficiently small. This property provides a practical criterion for assessing whether the iterates achieve local quadratic convergence according to the KKT residual in the numerical experiments.

\begin{lemma}\label{lemma:error-bound}
    Suppose that $(\bx,\blam)$ is a semistable solution of \eqref{KKT}. Then there exist positive constants $c>0$ and $\delta>0$ such that, for any $(x,\lam)$ satisfying $\Vert x - \bx \Vert + \Vert \lam - \blam \Vert \le \delta$,  the following error bound holds:
    \begin{equation}\label{error-bound}
        \Vert x - \bx \Vert + \Vert \lam - \blam \Vert \le c \left(\Vert F(x,\lam) \Vert + \Vert \min\left(\lam,-G(x) \right)\Vert \right).
    \end{equation}
\end{lemma}
\begin{proof}
    We proceed by contradiction. Assume that \eqref{error-bound} fails. Then there exists a sequence $\{(x^k,\lam^k) \}$ converging to $(\bx,\blam)$ such that
    \begin{equation}\label{F-lambda-G-order}
    \begin{aligned}
        \Vert F(x^k,\lam^k) \Vert& = o\left(\Vert x^k-\bx \Vert + \Vert \lam^k - \blam \Vert\right),\\
        \Vert \min ( \lam^k, -G(x^k) ) \Vert& = o\left(\Vert x^k - \bx \Vert + \Vert \lam^k - \blam \Vert \right).
    \end{aligned}
    \end{equation}
    By the Taylor expansion around $(\bx,\blam)$,
    \begin{equation}\label{error-bound-F}
        \begin{aligned}
            F(x^k,\lam^k) - F(\bx,\blam) &= 
            F(x^k,\lam^k) - F(x^k,\blam) + F(x^k,\blam) - F(\bx,\blam)\\
            & = E(x^k) \left( \lam^k - \blam  \right) + \J_x F(\bx,\blam)(x^k - \bx) + o\left( \Vert x^k - \bx \Vert \right)\\
            & = E(\bx) \left( \lam^k - \blam  \right) + \J_x F(\bx,\blam)(x^k - \bx) + o\left( \Vert x^k - \bx \Vert \right).
        \end{aligned}
    \end{equation}
    Dividing both sides of \eqref{error-bound-F} by $\Vert x^k-\bx\Vert + \Vert \lam^k - \blam \Vert$ and letting $k \rightarrow \infty$, we obtain a nontrivial limit $(\dx,\dlam)$ (after taking a subsequence if necessary) with
    $$
        \dx = \lim\limits_{k\rightarrow \infty} \frac{x^k-\bx}{\Vert x^k-\bx\Vert + \Vert \lam^k - \blam\Vert}\quad \text{and } 
        \dlam = \lim\limits_{k\rightarrow \infty} \frac{\lam^k-\blam}{\Vert x^k-\bx\Vert + \Vert \lam^k - \blam\Vert}
    $$
    such that
    \begin{equation}\label{error-bound-1-eq}
        \J_x F(\bx,\blam) \dx + E(\bx) \dlam = 0.
    \end{equation}
    Since $(x^k,\lam^k)\rightarrow (\bx,\blam)$, \eqref{F-lambda-G-order} implies that, for sufficiently large $k$,
    \begin{equation}\label{G-lam-order-I1-I3}
        \left\{\begin{array}{lll}
              |g^\nu_i(x^k) - g^\nu_i(\bx)| = o\left(\Vert x^k-\bx \Vert + \Vert \lam^k - \blam \Vert\right) & \text{for all}\ \nu\in[N], i\in I^\nu_+; \\
              |\lam^{k,\nu}_i| = o\left(\Vert x^k-\bx \Vert + \Vert \lam^k - \blam \Vert\right) & \text{for all}\ \nu\in[N], i\in I^\nu_-,
        \end{array}\right.
    \end{equation}
    and for all $\nu\in[N]$, $i\in I^\nu_0$,
    \begin{equation}\label{G-lam-order-I2}
    \begin{aligned}
         {\rm either}\ |\lam^{k,\nu}_i | & = o\left(\Vert x^k-\bx \Vert + \Vert \lam^k - \blam \Vert\right),\ g^\nu_i(x^k) \le 0,  \\
          {\rm or}\ | g^\nu_i(x^k) - g^\nu_i(\bx) | &= o\left(\Vert x^k-\bx \Vert + \Vert \lam^k - \blam \Vert\right),\ \lam^{k,\nu}_i \ge 0.
    \end{aligned}
    \end{equation}
    Dividing both sides of \eqref{G-lam-order-I1-I3} and \eqref{G-lam-order-I2} by $\Vert x^k-\bx\Vert + \Vert \lam^k - \blam \Vert$ and letting $k \rightarrow \infty$, we obtain  
    \begin{equation*}
        \left\{\begin{array}{lll}
               \J_x g^\nu_i(\bx) \dx = 0 & \text{for all}\ \nu\in[N],i\in I^\nu_+;\\
              0 \le \dlam^\nu_i \perp -\J_x g^\nu_i(\bx)\dx \ge 0 & \text{for all}\ \nu\in[N],i\in I^\nu_0;\\
              \dlam^\nu_i = 0 &  \text{for all}\ \nu\in[N],i\in I^\nu_-.
        \end{array}\right.
    \end{equation*}
    This together with \eqref{error-bound-1-eq} shows that $ (\dx,\dlam) \neq (0,0)$ solves the mixed LCP \eqref{unique},
    which contradicts the semistability of $(\bx,\blam)$ by Theorem~\ref{semistable}. The proof is complete.
\end{proof}

% In practical numerical instances, the global $(\alpha,\beta)$-monotonicity condition is not always satisfied. In such case, global convergence may fail, and the iterates may not converge at all. Consequently, even in a neighborhood of a semistable point, local quadratic convergence may no longer hold, and the additional hemistability condition is required to ensure the desired local behavior. 
For an NLP problem, the SMFCQ and SOSC at $(\bx,\blam)$ imply both semistability and hemistability \cite{bonnans1994local}, which still remain the weakest conditions guaranteeing the local quadratic convergence of SQP methods \cite{izmailov2015newton}. This follows from the fact that a semistable KKT pair $(\bx,\blam)$ with $\bx$  being a local minimizer of the NLP is also hemistable \cite{bonnans1994local}. However, the implication fails to hold for a GNEP.

\begin{example}
Consider the following generalized Nash equilibrium problem:
\begin{equation*}\label{exm-2}
\begin{array}{llll}
     \begin{array}{llll}
			{\rm Player\ 1} \quad &\quad \min\limits_{x^1} & \frac{1}{12}(x^1)^4 + x^1x^2 \\
			
			&\quad {\rm s.t.} & x^1 \ge 0,\\
			\\
			{\rm Player\ 2} \quad &\quad \min\limits_{x^2} & \frac{1}{2}(x^1+x^2-x^3)^2\\
	    \end{array} 
	    & \quad
       \begin{array}{lll}
            {\rm Player\ 3} \quad &\quad \min\limits_{x^3} & x^3 x^4 \\
			&\quad {\rm s.t.} & x^3 \ge 0,\\
			\\
			{\rm Player\ 4} \quad &\quad \min\limits_{x^4} & \frac{1}{2}(x^3+x^4-1)^2.\\ 
       \end{array}
\end{array}
\end{equation*}

The linearized KKT system at $(\tx,\tlam)$ is \begin{equation}\label{exm-2-KKT}
    \left\{\begin{array}{ll}
           \begin{pmatrix}
             (\tx^1)^2 & 1 & 0 & 0\\
             1 & 1 & -1 & 0\\
             0 & 0 & 0 & 1\\
             0 & 0 & 1 & 1
         \end{pmatrix}
         \begin{pmatrix}
         x^1\\
         x^2\\
         x^3\\
         x^4
         \end{pmatrix}
         + 
         \begin{pmatrix}
         -1 & 0\\
         0 & 0\\
         0 & -1\\
         0 & 0
         \end{pmatrix}
         \begin{pmatrix}
         \lam^1\\
         \lam^3
         \end{pmatrix}
         =
         \begin{pmatrix}
         \frac{2}{3}(\tx^1)^3\\
         0\\
         0\\
         1
         \end{pmatrix}\\
          0 \le 
         \begin{pmatrix}
             \lam^1\\
             \lam^3
         \end{pmatrix}
         \perp 
         \begin{pmatrix}
             1 & 0 & 0 & 0\\
             0 & 0 & 1 & 0
         \end{pmatrix}
         \begin{pmatrix}
             x^1\\
             x^2\\
             x^3\\
             x^4
         \end{pmatrix} \ge 0.
    \end{array}\right.
\end{equation}
From Theorem~\ref{semistable}, $(\bx,\blam) = (0,0,0,1,0,1)^\top$ is a semistable solution of \eqref{exm-2-KKT} where $\bx = (0,0,0,1)^\top$ is a GNE. It can be verified that the linearized KKT system  \eqref{exm-2-KKT} at any point $(\tx,\tlam) \in \R^4 \times \R^2$ admits at least one solution. However, for any $0< \tx^1 < 1$, \eqref{exm-2-KKT} has no solution in a neighborhood of $(\bx,\blam)$, implying that $(\bx,\blam)$ is not hemistable. 
\end{example}

The next theorem provides an exact characterization of the strong regularity that guarantees both the semistability and hemistability. By Theorem~\ref{thm:local quadratic convergence}, the SLCP method converges quadratically to a KKT pair $(\bx,\blam)$ at which the KKT system is strongly regular.

 Define the index set associated with the fixed point $(\bx,\blam)$ as
\begin{equation}\label{index-partition-set}
    \I = \left\{ \left(J_+,J_0,J_-\right) \, \left|  \  
    \begin{array}{ll}
           J_+ = \prod\limits_{\nu\in[N]} J^\nu_+, J_0 = \prod\limits_{\nu\in[N]} J^\nu_0, J_- = \prod\limits_{\nu\in[N]} J^\nu_-, \\ 
          J^\nu_+ \cup J^\nu_0 \cup J^\nu_- = [m_\nu],\ I^\nu_+ \subset J^\nu_+ \subset I^\nu_+ \cup I^\nu_0, \\
          \text{and }I^\nu_- \subset J^\nu_- \subset I^\nu_0 \cup I^\nu_-\ \text{ for all } \nu \in [N]
    \end{array}
    \right. \right\}.
\end{equation}

\begin{theorem}\label{thm:strong-regularity}
 The KKT system~\eqref{KKT} is strongly regular at $(\bx,\blam)$ if and only if, for any $(J_+,J_0,J_-)\in \I$, the following mixed LCP associated with $(J_+,J_0,J_-)$ admits $(\dx,\dlam) = (0,0)$ as its unique solution: 
 \begin{equation}\label{strong metric regularity-judge}
		\left\{\begin{array}{ll}
				\J_x F(\bx,\blam)^\top \dx - \J_x G(\bx)^\top \dlam  = 0, & \\
				\grad_{x^\nu} g^\nu_i (\bx)^\top \dx^\nu =0 & \text{for all}\ \nu \in[N], i\in J^\nu_+; \\
				\dlam^\nu_i \le 0, \grad_{x^\nu} g^\nu_i (\bx)^\top \dx^\nu \ge 0 & \text{for all}\ \nu\in [N], i\in J^\nu_0;\\
				\dlam^\nu_i = 0 & \text{for all}\ \nu\in [N], i\in J^\nu_-.
			\end{array}\right.
		\end{equation}
\end{theorem}

\begin{proof}
    Denote by $\L : \R^{n+m} \rightrightarrows \R^{n+m}$ the inverse of the set-valued mapping 
    \begin{equation*}
        \begin{pmatrix}
            F(\bx,\blam)\\
            -G(\bx)
        \end{pmatrix}
        + 
        \begin{pmatrix}
            \J_x F(\bx,\blam) & E(\bx)\\
            -\J_x G(\bx) & 0
        \end{pmatrix}
        \begin{pmatrix}
            x - \bx\\
            \lam - \blam
        \end{pmatrix}
        + 
        \N_{\R^n \times \R^m_+} (x,\lam).
    \end{equation*}
    By \cite[Theorem 1]{dontchev1996characterizations}, the strong regularity of \eqref{KKT} at $(\bx,\blam)$ is equivalent to the Aubin property \cite{aubin1984lipschitz} of $\L$ at $(0,0)$ for $(\bx,\blam)$. Using the Mordukhovich criterion \cite[Theorem 9.40]{rockafellar1998variational}, $\L$ has the Aubin property at $(0,0)$ for $(\bx,\blam)$ if and only if 
    \begin{equation}\label{prf-strong-regularity-1}
    \begin{aligned}
        &\left( \begin{pmatrix}
            - \J_x F(\bx,\blam)^\top & \J_x G(\bx)^\top\\
            -E(\bx)^\top & 0
        \end{pmatrix}
        \begin{pmatrix}
            \dx\\
            \dlam
        \end{pmatrix},
        \begin{pmatrix}
        -\dx\\
        -\dlam
        \end{pmatrix}
        \right)
        \\
        & \qquad \qquad \in 
        \N_{ {\rm gph}\, \N_{\R^n \times \R^m_+}} \left(  
        \begin{pmatrix}
        \bx\\
        \blam
        \end{pmatrix},
        \begin{pmatrix}
        -F(\bx,\blam)\\
        G(\bx)
        \end{pmatrix}
        \right)
         \implies 
        \begin{pmatrix}
        \dx\\
        \dlam
        \end{pmatrix}
        = 
        \begin{pmatrix}
            0\\
            0
        \end{pmatrix},
    \end{aligned}
    \end{equation}
    where ${\rm gph}\,\N_{\R^n \times \R^m_+} := \left\{ (y,z,u,v)\in \R^{2n+2m} \mid (u,v) \in \N_{\R^n \times \R^m_+} (y,z)  \right\}$.
    Combined with the characterization of the set-valued mapping $\N_{ {\rm gph}\, \N_{\R^n \times \R^m_+}}$ (see \cite{mordukhovich2007coderivative}), \eqref{prf-strong-regularity-1} reduces to the implication \eqref{strong metric regularity-judge}, which completes the proof.
\end{proof}

\begin{remark}
We omit the proof details and refer the interested reader to \cite{diao2025stability,dontchev2009implicit}. Although the model considered in \cite{diao2025stability} is a standard NEP, the arguments used to establish strong regularity apply directly in the present context. It is worth noting that strong regularity imposes stronger conditions on the solution than either semistability or hemistability.
\end{remark}

To illustrate Theorem~\ref{thm:strong-regularity}, Appendix~\ref{appendix:further analysis for the internet switching model} provides a detailed analysis of the strong regularity for Example~\ref{exm-internet} (see Proposition~\ref{prop:exm-strong-regularity}).

% We conclude this section by presenting a unified summary of the convergence results for the SLCP method, obtained by combining \ref{thm:boundedness}, \ref{thm:local quadratic convergence}, \ref{coro:monotone-implies-semistable}, and \ref{thm:strong-regularity}.

% \begin{corollary}
% Let $(\bx,\blam)$ be a KKT pair of \eqref{KKT}. Then the following statements hold:
% \begin{enumerate}
%     \item For any starting point $(x^0,\lambda^0)\in \R^n\times \R^m_+$, suppose that the conditions $(i)-(iii)$ from \ref{thm:boundedness} hold, the number of accumulation points remains one (which is always more than one). Then there exists a
% \end{enumerate}
% \end{corollary}

% \vspace{5em}

\section{Experimental results}
\label{sec:experiments}

In this section, we compare Algorithm~\ref{algslcp} with three representative algorithms. The first two algorithms, introduced in \cite{dreves2011solution}, are based on the interior-point method (IPM) proposed in \cite{monteiro1999potential} and the semismooth-like minimization method (SMM) developed in \cite{de1996semismooth,de2000theoretical,facchinei2003finite}, respectively. The third algorithm \cite{kanzow2016augmented} is based on the ALM, with further discussions presented in \cite{bueno2019optimality,jordan2023first,kanzow2018augmented,kim2023new}. We regard these algorithms as classical benchmarks. The test problems are drawn from \cite{dreves2011solution,facchinei2010penalty} and the references therein. These problems have also been tested in
\cite{dreves2011solution,kanzow2016augmented}, which allows for a direct and meaningful comparison of the numerical performance.

% , which provide a relatively comprehensive and widely adopted test set. A total of 62 test instances are considered, where different initialization points of the same problem are treated as separate instances, since they are explicitly used \cite{dreves2011solution,kanzow2016augmented} to compare the performance of the algorithms.

All computational results were obtained on a Windows 11 personal computer equipped with an Intel Ultra 7 155H processor (16 cores, 22 threads, 4.8 GHz) and 32 GB of RAM. All algorithms were implemented in MATLAB~2022a and were terminated once
% To clearly measure the optimality, feasibility, and complementarity conditions of the iterates generated by each algorithm, we define three residuals, denoted by $R_o$, $R_g$, and $R_c$ respectively, as follows:
% \begin{equation}
%     \begin{aligned}
%      R_o = \Vert F(x,\lam) \Vert_{\infty},\quad
%      R_g = \Vert G(x) \Vert_{\infty}, \quad
%      R_c = \max_{\nu\in[N],i\in [m_\nu]} | \lam^\nu_i g^\nu_i (x) |.
%     \end{aligned}
% \end{equation}
$$
    \max\left\{  \Vert F(x,\lam) \Vert_{\infty}, \Vert \left( G(x) \right)_+ \Vert_{\infty}, \max_{\nu\in[N],i\in [m_\nu]} | \lam^\nu_i g^\nu_i (x) | \right\} \le {\rm tol},
$$
where ${\rm tol} = 10^{-7}$.  The time limit is set to 30 minutes. 
% Since the IPM does not involve any subproblem structure, this limit applies to the total number of iterations. In contrast, for both ALM and the SLCP method, the limit refers to the number of outer iterations.

\subsection{Solution of the subproblems}
In practice, solving the LC subproblems constitutes the most computationally expensive component of the proposed method. Consequently, the overall efficiency of the SLCP method is strongly influenced by the choice of the algorithm for subproblems, the extent to which problem structure can be exploited, and whether the LC subproblems are solved inexactly. We adopt the following techniques.
\begin{enumerate}
    \item Practical LC subproblems often exhibit exploitable structure. For example, in many test instances, the constraint $x^\nu \in \R^{n_\nu}_{+}$ is present for all $\nu \in [N]$. In such cases, the LC subproblem at each iteration $k$ can be reformulated as a standard LCP:
\begin{equation*}
\begin{aligned}
   & 0\le \begin{pmatrix}
        p^k + x^k\\
        q^k_\alpha + \lam^k_\alpha
    \end{pmatrix} \perp 
    \begin{pmatrix}
        \J_x F(x^k,\lam^k) & E_\alpha (x^k)\\
        -\J_x G_\alpha (x^k) & 0
    \end{pmatrix}
    \begin{pmatrix}
        p^k + x^k\\
        q^k_\alpha + \lam^k_\alpha
    \end{pmatrix}\\
   & \qquad \qquad \qquad \qquad \qquad+ 
    \begin{pmatrix}
    F(x^k,\lam^k) - E(x^k)\lam^k - \J_x F(x^k,\lam^k)x^k\\
    -G_\alpha(x^k) + \J_x G_\alpha (x^k) x^k
    \end{pmatrix} \ge 0.
\end{aligned}
\end{equation*}
Here, $\alpha$ refers to the remaining constraints (i.e., those other than $x^\nu \in \mathbb{R}^{n_\nu}_+$).  
This reformulation reduces the problem dimension and therefore lowers the computational cost. 
\item For LC subproblems with $n\le 100$, we minimize the Fischer-Burmeister function \cite{fischer1995newton} using a Levenberg--Marquardt trust-region scheme \cite{nocedal2006numerical} with $\sigma \in [10^{-5},10^{-3}]$ and an update scaling factor equal to $5$. We terminate the subproblem algorithm when the residual falls below $10^{-8}$.
\item For larger LC subproblems with $n>100$, we employ an IPM similar to that used in \cite{dreves2011solution}. Moreover, we allow inexact solutions to the subproblems to further reduce the computational cost. Specifically, the subproblem algorithm is terminated when the residual falls below $\max\{ 10^{-8}, 10^{-{\rm iter}} \}$, where ${\rm iter}$ denotes the outer iteration counter.
\end{enumerate}
Since the LC subproblems are equivalent to affine GNEPs when the GNEP is player-convex, a wide range of efficient algorithms are available for their solution; see, e.g., \cite{cottle2009linear,dreves2014finding,schiro2013solution}. This suggests that the LC subproblems have considerable computational potential. In our experiments, we adopt both the Levenberg--Marquardt type method and the IPM. The IPM was used in \cite{dreves2011solution} to directly solve the original problem, and the Levenberg--Marquardt type method was employed in \cite{kanzow2016augmented} to solve the ALM subproblems. Using these methods for the LC subproblems in our framework enables a more meaningful evaluation of how linearization contributes to improved computational performance.

\subsection{Performance profiles}
 We employ performance profiles \cite{dolan2002benchmarking} to compare the numerical performance of different algorithms. Denote the set of algorithms and benchmark problems by  $\mathcal{A}$ and $\mathcal{B}$, respectively. Let $s_{a,b}$ be a performance measure, such as runtime, gradient evaluations, or Hessian evaluations. If algorithm $a$ fails to solve benchmark problem $b$, set $s_{a,b} = + \infty$.  The performance ratio is defined by
$$
    r_{a,b} = \frac{s_{a,b}}{\min\limits_{a\in \mathcal{A}} \left\{ s_{a,b} \right\}}.
$$
The proportion of problems for which algorithm $a\in \mathcal{A}$ achieves a performance ratio within a factor $\tau \in \R_+$ of the best performance is given by
$$
    \rho_{a} (\tau) = \frac{1}{| \mathcal{B} |} {\rm size}\left\{ b\in \mathcal{B} \mid r_{a,b}\le \tau \right\}.
$$
Thus, the performance profile represents the cumulative distribution function of the performance ratios for each algorithm. An algorithm with a higher value of $\rho_a (\tau)$
at a given $\tau$ is considered more efficient.

We now present performance profiles comparing the four algorithms on the test problems. The comparison focuses on the number of gradient evaluations of $\theta^\nu$ and $G$, the number of Hessian evaluations of $L^\nu$, and the total time required to reach the desired accuracy. Detailed numerical results are summarized in Figure~\ref{fig:performance_profiles}, while the complete numerical results are
listed in Tables~\ref{alg:table-time-1}--\ref{alg:table-grad-Hessian-1}. The main conclusions are as follows:
\begin{enumerate}
    \item The SLCP method successfully solves all test problems, whereas the other algorithms fail on some instances. Moreover, the SLCP method achieves the shortest runtime on approximately $85\%$ of the problems. All four methods are capable of attaining the target accuracy of $10^{-7}$
  on the majority of problems, indicating a relatively high level of accuracy. Note that, although the IPM implementation in \cite{dreves2011solution} did not aim for this level of accuracy, the method itself is capable of achieving such accuracy in principle.
  \item The SLCP method not only demonstrates high efficiency in terms of computational time, but also exhibits a significant advantage in the number of Hessian and gradient evaluations. This performance can be attributed in part to the fact that almost half of the test problems are instances of affine GNEP. For these instances, the subproblems of the SLCP method coincide exactly with the original problems, and thus the Hessian and gradient only need to be evaluated once at the beginning, with no redundant computations thereafter. This highlights the SLCP method's ability to exploit the problem structure.
Another contributing factor is related to the inherent nature of the SQP method. SQP-based solvers such as SNOPT \cite{gill2005snopt} and NLPQL \cite{schittkowski1986nlpql}, typically require fewer gradient evaluations \cite{gill2005snopt} compared to solvers based on alternative frameworks such as MINOS \cite{murtagh1978large,murtagh1982projected} and CONOPT \cite{drud1985conopt}. Since the SLCP method is essentially like SQP methods, the
fewer gradient and Hessian evaluations observed for the SLCP method can be seen as a direct consequence of its SQP-style design.
\item For the ALM, achieving high accuracy solutions typically requires increasingly accurate solutions of the subproblems as the iterates approach optimality. However, the ALM subproblems are nonlinear and not twice continuously differentiable. Solving these subproblems using the Levenberg--Marquardt type method entails significantly more subproblem gradient and Hessian evaluations, as illustrated in Figure~\ref{fig:JL}.
In contrast, the subproblems in the SLCP method are affine GNEPs, which are substantially easier to solve than the ALM subproblems, especially in high-dimensional settings. This structural simplicity leads to a significant reduction in the number of subproblem iterations, as evidenced in Table~\ref{alg:table-grad-Hessian-1}.

\end{enumerate}

\begin{figure}[htbp!]
  \centering

  \begin{subfigure}[b]{0.42\textwidth}
    \centering
    \includegraphics[width=\textwidth]{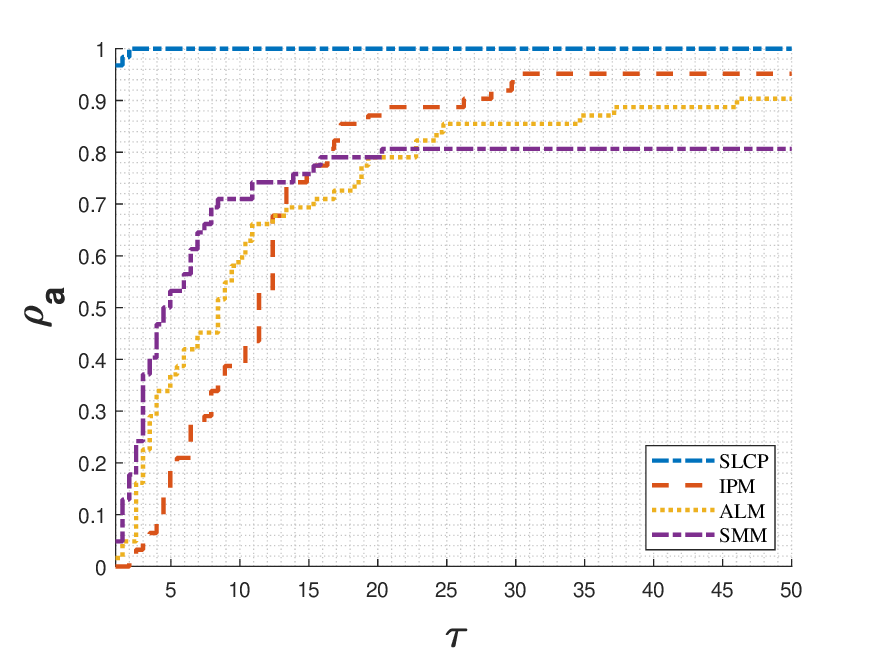}
    \caption{Evaluations of $\J_{x^\nu} \theta^\nu$ and $\J_x g$}
    \label{fig:grad}
  \end{subfigure}
  \hfill
  \begin{subfigure}[b]{0.42\textwidth}
    \centering
    \includegraphics[width=\textwidth]{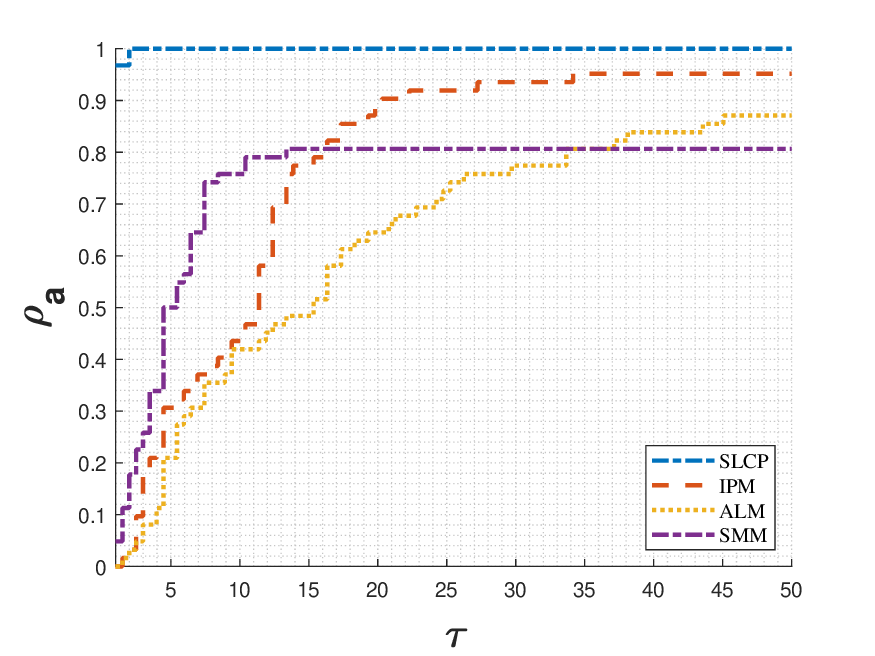}
    \caption{Evaluations of $\J_x F$}
    \label{fig:JL}
  \end{subfigure}
  
  \begin{subfigure}[b]{0.42\textwidth}
    \centering
    \includegraphics[width=\textwidth]{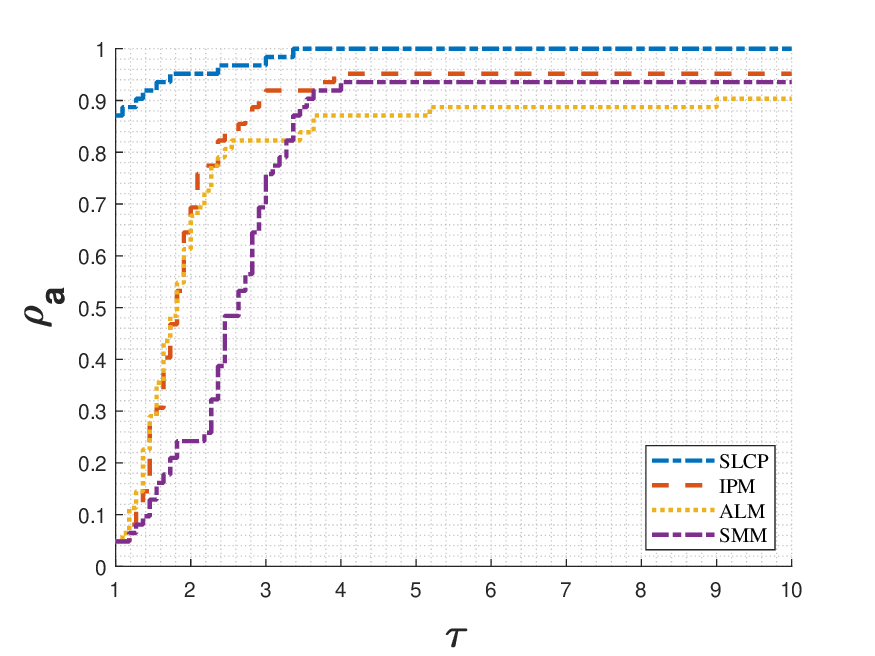}
    \caption{Runtime}
    \label{fig:time}
  \end{subfigure}

  \caption{Performance profiles of SLCP, IPM, ALM, and SMM on test problems}
  \label{fig:performance_profiles}
\end{figure}

\subsection{Convergence behavior on representative problems}
We now present the convergence results of the four algorithms on two representative problems. The first problem is the internet switching model (Problem A1 in Table~\ref{alg:table-time-1}), and the second is the classical Arrow and Debreu's competitive economy model (Problem A10a in Table~\ref{alg:table-time-2}) from \cite{arrow1954existence}.

For each problem, all four algorithms converge to the same GNE $\bx$ with $\Vert \bx \Vert_2\le 100$. We evaluate the algorithms starting from two distant initial points $10^3 \cdot \mathbf{1}_n$ and $10^6 \cdot \mathbf{1}_n$ to examine the overall convergence behavior, where $\mathbf{1}_n$ denotes the vector of ones in $\mathbb{R}^n$. Among the four, the SLCP method and the SMM are able to consistently converge from both starting points. 

As shown in Figure~\ref{fig:A1}, the SLCP method exhibits local quadratic convergence toward the KKT pair, which is consistent with our theoretical results in Proposition~\ref{prop:exm-strong-regularity} and Lemma~\ref{lemma:error-bound}. Moreover, such quadratic convergence behavior is frequently observed in practice, as illustrated in Figure~\ref{fig:A10a}. The ALM demonstrates good global convergence properties in Figure~\ref{fig:A1}, provided that its subproblems are solved with sufficient accuracy, in line with the conclusions drawn in \cite{kanzow2016augmented}. However, when the subproblem solutions are solved inaccurately, as in Figure~\ref{fig:A10a}, the ALM may fail to converge. The IPM  generally requires more iterations, and does not exhibit the same local quadratic convergence rate in the neighborhood of the solution as the SLCP method and the SMM.

\begin{figure}[htbp]
  \centering
  \begin{subfigure}[t]{0.45\textwidth}
    \centering
    \includegraphics[width=\linewidth]{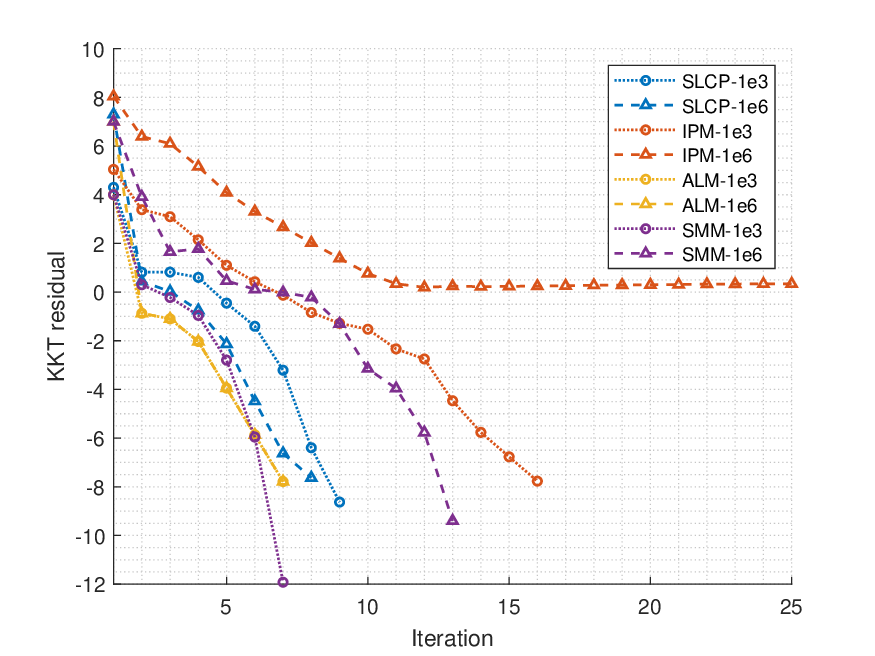}
    \caption{Internet switching model (A1)}
    \label{fig:A1}
  \end{subfigure}
  \hfill
  \begin{subfigure}[t]{0.45\textwidth}
    \centering
    \includegraphics[width=\linewidth]{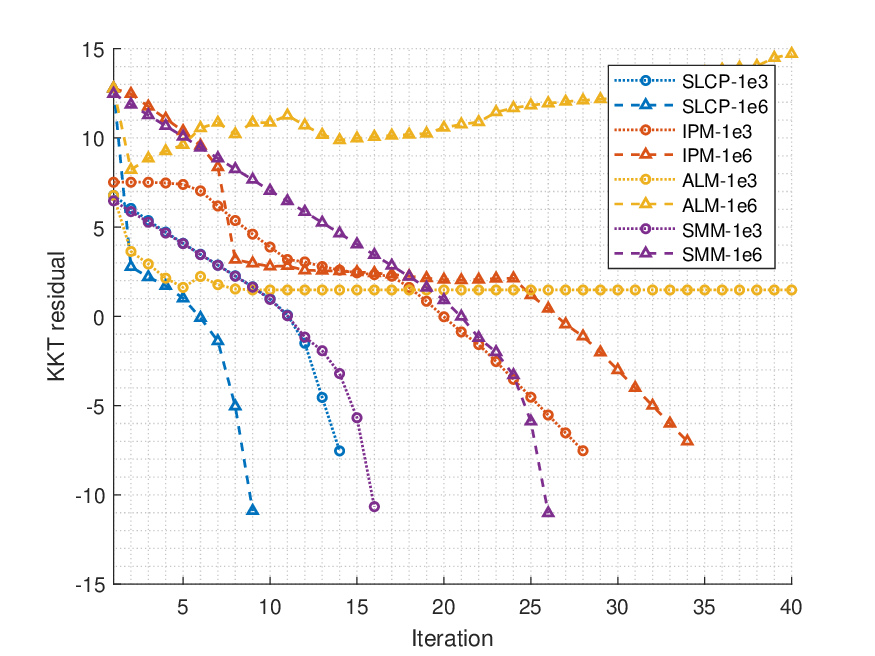}
    \caption{Arrow and Debreu's competitive economy model (A10a)}
    \label{fig:A10a}
  \end{subfigure}
  \caption{Global convergence behavior of SLCP, IPM, ALM, and SMM on two representative test instances}
  \label{fig:twoside}
\end{figure}

\section{Conclusions}
\label{sec:conclusions}

This paper has proposed a sequential linear complementarity problem  method for solving generalized Nash equilibrium problems. Under suitable assumptions, the proposed method has been shown to achieve both global convergence and local quadratic convergence. When the GNEP reduces to a standard nonlinear programming problem, the convergence conditions coincide with those of the classical SQP method with exact Lagrangian Hessians. Numerical results have demonstrated that the SLCP method outperforms the existing IPM, SMM, and ALM on the test instances.

Several promising directions for further investigation remain. Motivated by the development of modern SQP schemes, it is natural to investigate whether the SLCP subproblems can incorporate approximate Hessians instead of exact second-order information, while still maintaining global convergence. Such extensions could significantly enhance the computational efficiency of the method and are closely related to possible refinements of the novel merit function proposed in this work. Moreover, the integration of efficient solvers for the linearized subproblems, such as Lemke's method considered in \cite{schiro2013solution}, offers another promising research direction.

Overall, the proposed SLCP method can be viewed as a natural extension of the SQP framework from classical nonlinear programming to the setting of GNEPs. Combining the linearization strategy developed here with augmented Lagrangian techniques in \cite{kanzow2016augmented} provides a promising direction for addressing the issues discussed above.

% \appendix
% \section{An example appendix} 
% \lipsum[71]

% \begin{lemma}
% Test Lemma.
% \end{lemma}

% \section*{Acknowledgments}
% We would like to acknowledge the assistance of volunteers in putting
% together this example manuscript and supplement.

\appendix

\section{Table of Numerical Results}\label{appendix:numerical results}
This appendix provides additional tables with further details on the numerical results reported in Section~\ref{sec:experiments}. Specifically,
\begin{enumerate}
    \item Tables~\ref{alg:table-time-1}-\ref{alg:table-time-2} present the runtime comparison for all test problems. Entries marked ``--'' indicate that the algorithm failed to obtain a solution within the time limit of 30 minutes. For each problem, the fastest algorithm is highlighted in bold. All algorithms start from the initial point $x^0\cdot \mathbf{1}_n$. 
    \item Table~\ref{alg:table-grad-Hessian-1} summarizes the number of gradient and Hessian evaluations performed by each algorithm. Here, ``Grad'' denotes the total number of gradient evaluations of $\theta^\nu$ and $g^\nu$, while ``Hess'' represents the number of evaluations of the Hessian $\J_x F$.
\end{enumerate}

\begin{table}[htbp]
\caption{Runtime of SLCP, ALM, IPM, and SMM on test problems.}
\label{alg:table-time-1}
\centering
\setlength{\tabcolsep}{10pt}
\renewcommand{\arraystretch}{1.15}  
\begin{tabular}{cccccc}
\toprule
Problem & $x^0$ & SLCP & ALM & IPM & SMM \\
\midrule
A1                            & 0.1 & \textbf{0.024} & 0.033          & 0.033          & 0.037                     \\
                              & 1   & \textbf{0.025} & 0.036          & 0.040          & 0.058                     \\
                              & 10  & \textbf{0.025} & 0.038          & 0.039          & 0.057                     \\ \hline
A2                            & 0.1 & \textbf{0.012} & 0.027          & 0.046          & 0.033                     \\
                              & 1   & \textbf{0.027} & 0.065          & 0.047          & 0.066                     \\
                              & 10  & \textbf{0.027} & 0.062          & 0.042          & 0.065                     \\ \hline
A3                            & 0   & \textbf{0.018} & 0.029          & 0.053          & 0.047                     \\
                              & 1   & \textbf{0.016} & 0.030          & 0.045          & 0.046                     \\
                              & 10  & \textbf{0.019} & 0.030          & 0.044          & 0.065                     \\ \hline
A4                            & 0   & \textbf{0.040} & 0.054          & 0.049          & 0.064                     \\
                              & 1   & \textbf{0.033} & 0.051          & 0.058          & 0.046                     \\
                              & 10  & \textbf{0.042} & 0.053          & 0.056          & 0.070                     \\ \hline
A5                            & 0   & \textbf{0.020} & 0.045          & 0.047          & 0.067                     \\
                              & 1   & \textbf{0.018} & 0.045          & 0.043          & 0.064                     \\
                              & 10  & \textbf{0.020} & 0.043          & 0.041          & 0.066                     \\ 
 
\bottomrule
\end{tabular}
\end{table}

\begin{table}[htbp]
\caption{Runtime of SLCP, ALM, IPM, and SMM on  test problems.}
\label{alg:table-time-2}
\centering
\setlength{\tabcolsep}{10pt}
\renewcommand{\arraystretch}{1.15}  
\begin{tabular}{cccccc}
\toprule
Problem & $x^0$ & SLCP & ALM & IPM & SMM \\
\midrule
A6                            & 0   & \textbf{0.029} & 0.058          & 0.042          & 0.068                     \\
                              & 1   & \textbf{0.026} & 0.058          & 0.042          & 0.072                     \\
                              & 10  & \textbf{0.032} & 0.057          & 0.047          & -                         \\ \hline
A7                            & 0   & \textbf{0.030} & 0.055          & 0.056          & 0.081                     \\
                              & 1   & \textbf{0.031} & 0.054          & 0.059          & 0.043                     \\
                              & 10  & \textbf{0.039} & 0.056          & 0.055          & 0.069                     \\ \hline
A8                            & 0   & \textbf{0.021} & -          & -              & 0.032                     \\
                              & 1   & \textbf{0.019} & 0.021          & -              & 0.066                     \\
                              & 10  & \textbf{0.020} & 0.027          & -              & 0.058                     \\ \hline
A9a                           & 0   & \textbf{0.086} & 0.312          & 0.110          & 0.097                     \\ \hline
A9b                           & 0   & 0.150          & 0.565          & 0.190          & \textbf{0.110}            \\ \hline
A10a                          & 0   & \textbf{0.069} & 0.233          & 0.085          & 0.088                     \\ \hline
A10b                          & 0   & 0.182          & 1.076          & \textbf{0.120} & 4.400                     \\ \hline
A10c                          & 0   & 0.481          & 9.420          & 0.240          & \textbf{0.210}            \\ \hline
A10d                          & 0   & 1.141          & 189.498        & \textbf{0.390} & -                         \\ \hline
A10e                          & 0   & 2.978          & 582.917        & \textbf{0.890} & 66.000                    \\ \hline
A11                           & 0   & \textbf{0.016} & 0.031          & 0.032          & 0.053                     \\
                              & 1   & \textbf{0.017} & 0.028          & 0.031          & 0.038                     \\
                              & 10  & \textbf{0.017} & 0.028          & 0.032          & 0.055                     \\ \hline
A12                           & 0   & \textbf{0.013} & 0.020          & 0.034          & 0.038                     \\
                              & 1   & \textbf{0.012} & 0.020          & 0.034          & 0.039                     \\
                              & 10  & \textbf{0.014} & 0.026          & 0.033          & 0.056                     \\ \hline
A13                           & 0   & \textbf{0.020} & 0.038          & 0.054          & 0.060                     \\
                              & 1   & \textbf{0.022} & 0.036          & 0.040          & 0.060                     \\
                              & 10  & \textbf{0.021} & 0.038          & 0.039          & 0.060                     \\ \hline 
A14                           & 0.1 & \textbf{0.025} & 0.026          & 0.036          & 0.030                     \\
                              & 1   & 0.033          & \textbf{0.027} & 0.039          & 0.057                     \\
                              & 10  & 0.028          & \textbf{0.026} & 0.038          & 0.059                     \\  \hline
A15                           & 0   & \textbf{0.020} & 0.031          & 0.041          & 0.061                     \\
                              & 1   & \textbf{0.020} & 0.026          & 0.041          & 0.063                     \\
                              & 10  & \textbf{0.018} & 0.029          & 0.038          & 0.044            \\ \hline
A16a                          & 10  & \textbf{0.025} & 0.029          & 0.042          & 0.058                     \\ \hline
A16b                          & 10  & \textbf{0.025} & 0.034          & 0.042          & 0.060                     \\ \hline
A16c                          & 10  & \textbf{0.027} & 0.037          & 0.043          & 0.059                     \\ \hline
A16d                          & 10  & \textbf{0.028} & 0.037          & 0.039          & 0.062                     \\ \hline
A17                           & 0   & \textbf{0.021} & 0.031          & 0.036          & 0.058                     \\
                              & 1   & \textbf{0.019} & 0.034          & 0.039          & 0.057                     \\
                              & 10  & \textbf{0.019} & 0.037          & 0.035          & 0.056                     \\ \hline
A18                           & 0   & \textbf{0.031} & 0.060          & 0.060          & 0.074                     \\
                              & 1   & \textbf{0.028} & 0.061          & 0.056          & 0.072                     \\
                              & 10  & \textbf{0.029} & 0.059          & 0.053          & 0.070                     \\ \hline
Harker                        & 1   & \textbf{0.019} & \textbf{0.019} & 0.026          & 0.053                     \\ \hline
Heu                           & 1   & \textbf{0.042} & -        & 0.075          & 0.069                     \\ \hline
NTF1                          & 1   & \textbf{0.019} & 0.024          & 0.026          &     0.053                  \\ \hline
NTF2                          & 1   & \textbf{0.021} & 0.023          & 0.023          &      0.054               \\ \hline
Spam                          & 1   & 12.852         & 27.805         & 27.410         &     \textbf{7.700}          \\ \hline
Lob                           & 1   & \textbf{0.028} & -          & 0.045          &     0.049               \\ 
\bottomrule
\end{tabular}
\end{table}

\begin{table}[htbp]
\centering

\caption{Gradient and Hessian evaluations of SLCP, ALM, IPM, and SMM on test problems.}
\label{alg:table-grad-Hessian-1}
\begin{tabular}{ccccccccccccc}
\toprule
\multirow{2}{*}{\makecell{Problem}} & \multirow{2}{*}{\makecell{$x^0$}} &
\multicolumn{2}{c}{SLCP} &
\multicolumn{2}{c}{ALM} &
\multicolumn{2}{c}{IPM} &
\multicolumn{2}{c}{SMM} 
\\
\cmidrule(lr){3-4} \cmidrule(lr){5-6} \cmidrule(lr){7-8} \cmidrule(lr){9-10}
& & Grad & Hess & Grad & Hess & Grad & Hess & Grad & Hess \\
\midrule
A1      & 0.1 & 8    & 3       & 26   & 13                           & 48                       & 12                           & 50                       & 8                            \\
        & 1   & 14   & 6       & 32   & 16                           & 60                       & 15                           & 18                       & 8                            \\
        & 10  & 14   & 6       & 42   & 21                           & 68                       & 17                           & 14                       & 6                            \\ \hline
A2      & 0.1 & 4    & 1       & 8    & 4                            & 112                      & 27                           & 62                       & 13                           \\
        & 1   & 10   & 4       & 346  & 173                          & 148                      & 37                           & 62                       & 13                           \\
        & 10  & 10   & 4       & 458  & 229                          & 120                      & 30                           & 68                       & 16                           \\ \hline
A3      & 0   & 4    & 1       & 10   & 5                            & 48                       & 12                           & 10                       & 4                            \\
        & 1   & 4    & 1       & 10   & 5                            & 48                       & 12                           & 10                       & 4                            \\
        & 10  & 4    & 1       & 10   & 5                            & 64                       & 16                           & 28                       & 10                           \\ \hline
A4      & 0   & 6    & 2       & 134  & 67                           & 52                       & 27                           & 12                       & 5                            \\
        & 1   & 4    & 1       & 148  & 74                           & 66                       & 34                           & 10                       & 4                            \\
        & 10  & 10   & 4       & 166  & 83                           & 88                       & 45                           & 58                       & 17                           \\ \hline
A5      & 0   & 4    & 1       & 32   & 16                           & 48                       & 12                           & 14                       & 6                            \\
        & 1   & 4    & 1       & 32   & 16                           & 48                       & 12                           & 16                       & 7                            \\
        & 10  & 4    & 1       & 42   & 21                           & 52                       & 13                           & 26                       & 10                           \\ \hline
A6      & 0   & 8    & 3       & 146  & 73                           & 80                       & 20                           & 28                       & 10                           \\
        & 1   & 8    & 3       & 150  & 75                           & 80                       & 20                           & 60                       & 15                           \\
        & 10  & 18   & 8       & 194  & 97                           & 140                      & 35                           & -                      & -                          \\ \hline
A7      & 0   & 6    & 2       & 60   & 30                           & 156                      & 39                           & 82                       & 14                           \\
        & 1   & 6    & 2       & 60   & 30                           & 176                      & 44                           & 64                       & 11                           \\
        & 10  & 22   & 4       & 72   & 36                           & 164                      & 41                           & 42                       & 12                           \\ \hline
A8      & 0   & 4    & 1       & -  & -                          & -                      & -                          & 4408                     & 516                          \\ 
        & 1   & 4    & 1       & 8    & 4                            & -                      & -                          & 2002                     & 236                          \\
        & 10  & 4    & 1       & 22   & 11                           & -                      & -                          & 4284                     & 502                          \\ \hline
A9a     & 0   & 54   & 17      & 318  & 159                          & 96                       & 24                           & 70                       & 23                           \\ \hline
A9b     & 0   & 58   & 17      & 292  & 146                          & 146                      & 36                           & 66                       & 26                           \\ \hline
A10a    & 0   & 22   & 9       & 530  & 265                          & 80                       & 20                           & 50                       & 15                           \\ \hline
A10b    & 0   & 30   & 11      & 740  & 370                          & 68                       & 17                           & 55278                    & 1743                         \\ \hline
A10c    & 0   & 28   & 9       & 3162 & 1581                         & 148                      & 35                           & 426                      & 39                           \\ \hline
A10d    & 0   & 24   & 8       & 3048 & 1524                         & 88                       & 22                           & -                      & -                          \\ \hline
A10e    & 0   & 28   & 8       & 3160 & 1580                         & 96                       & 24                           & 568                      & 48                           \\ \hline

A11     & 0   & 4    & 1       & 32   & 16                           & 44                       & 11                           & 12                       & 5                            \\
        & 1   & 4    & 1       & 32   & 16                           & 44                       & 11                           & 10                       & 4                            \\
        & 10  & 4    & 1       & 34   & 17                           & 44                       & 11                           & 12                       & 5                            \\ \hline
A12     & 0   & 4    & 1       & 8    & 4                            & 44                       & 11                           & 10                       & 4                            \\
        & 1   & 4    & 1       & 8    & 4                            & 44                       & 11                           & 10                       & 4                            \\
        & 10  & 4    & 1       & 8    & 4                            & 44                       & 11                           & 18                       & 7                            \\ \hline
A13     & 0   & 4    & 1       & 48   & 24                           & 40                       & 10                           & 42                       & 8                            \\
        & 1   & 4    & 1       & 34   & 17                           & 52                       & 13                           & 32                       & 7                            \\
        & 10  & 4    & 1       & 36   & 18                           & 48                       & 12                           & 30                       & 7                            \\ \hline
A14     & 0.1 & 8    & 3       & 16   & 8                            & 48                       & 12                           & 48                       & 7                            \\
        & 1   & 20   & 9       & 12   & 6                            & 48                       & 12                           & 12                       & 5                            \\
        & 10  & 20   & 9       & 20   & 10                           & 56                       & 14                           & 14                       & 6                            \\ \hline
A15     & 0   & 4    & 1       & 14   & 7                            & 52                       & 13                           & 18                       & 6                            \\
        & 1   & 4    & 1       & 14   & 7                            & 52                       & 13                           & 14                       & 6                            \\
        & 10  & 4    & 1       & 14   & 7                            & 48                       & 12                           & 10                       & 4                            \\ \hline
A16a    & 10  & 8    & 3       & 54   & 27                           & 60                       & 15                           & 14                       & 6                            \\ \hline
A16b    & 10  & 10   & 4       & 48   & 24                           & 64                       & 16                           & 22                       & 7                            \\ \hline
A16c    & 10  & 10   & 4       & 46   & 23                           & 64                       & 16                           & 16                       & 7                            \\ \hline
A16d    & 10  & 12   & 5       & 38   & 19                           & 56                       & 14                           & 48                       & 11                           \\ \hline
A17     & 0   & 4    & 1       & 38   & 19                           & 68                       & 17                           & 626                      & 62                           \\
        & 1   & 4    & 1       & 52   & 26                           & 68                       & 17                           & 1726                     & 191                          \\
        & 10  & 4    & 1       & 90   & 45                           & 64                       & 16                           & 6874                     & 694                          \\ \hline
A18     & 0   & 4    & 1       & 76   & 38                           & 76                       & 19                           & 14220                    & 3554                         \\
        & 1   & 4    & 1       & 74   & 37                           & 80                       & 20                           & 11386                    & 2848                         \\
        & 10  & 6    & 2       & 90   & 45                           & 72                       & 18                           & 18824                    & 4708                         \\ \hline
Harker  & 1   & 4    & 1       & 10   & 5                            & 60                       & 15                           & 14                       & 6                            \\ \hline
Heu     & 1   & 28   & 13      & -  & -                          & 842                      & 33                           & 32                       & 14                           \\ \hline
NTF1    & 1   & 4    & 1       & 26   & 13                           & 48                       & 12                                    & 22      & 7                   \\ \hline
NTF2    & 1   & 10   & 4       & 94   & 47                           & 48                       & 12    & 12                       & 5                                                    \\ \hline
Spam    & 1   & 6    & 2       & 8    & 4                            & 44                       & 11     & 14                       & 6                                                  \\ \hline
Lob     & 1   & 4    & 1       & -  & -                          & 32                       & 8       & 10                       & 3                                                \\ 
        \bottomrule
\end{tabular}
\end{table}

\section{Additional analysis for the internet switching model}\label{appendix:further analysis for the internet switching model}

As a direct consequence of Theorems~\ref{thm-subproblem-solvability-jointly} and \ref{thm-subproblem-general-solvability},  we analyze the subproblems arising from Example~\ref{exm-internet}. 

\begin{proposition}\label{exm-subproblem-solvability}
Suppose that the lower bounds $l_\nu$ in Example~\ref{exm-internet} satisfy the feasibility condition $\sum_{\nu\in[N]} l_\nu \le B$ . Then, under either of the following conditions, the LC subproblems of Example~\ref{exm-internet} at any $(x,\lam)\in \R^n_{++} \times \R^m_+$ admit a solution:
\begin{enumerate}
    \item All users are constrained by the buffer capacity, i.e., $[N_1]=[N]$;
    \item The index set $[N_1] \neq [N]$, $L_\nu < +\infty$ for all $\nu \in [N] \setminus [N_1]$, and there exists $\hnu \in [N_1]$ such that 
    \begin{equation}\label{coro-subproblem-solvability-equation}
        l_{\hnu} + \sum\limits_{ \mu\in [N] \setminus  [N_1] } L_\mu + \sum\limits_{\mu \in [N_1] \setminus \{ \hnu\} }\min\left\{ \left( B - \sum\limits_{\omega \in [N] \setminus \{ \mu \} } l_\omega \right),L_\mu \right\} \le B.
    \end{equation}
\end{enumerate}
\end{proposition}

\begin{proof}
   Recall that $S = \sum_{\nu \in [N]} x^\nu$. Following \eqref{subproblem-restricted}, consider the restricted GNE subproblem associated with Example~\ref{exm-internet} at $(x,\lam)$, where each player $\nu$ solves 
    \begin{equation*}
    \begin{array}{ll}
		\min\limits_{p^\nu \in \R} &  \left(\frac{1}{B} - \frac{S - x^\nu}{S^2}\right)p^\nu + \frac{S-x^\nu}{S^3}(p^\nu)^2 + \sum\limits_{ \mu \in [N] \setminus \{ \nu \}}\frac{S - 2x^\nu }{S^3}p^\nu p^\mu\\
		\text{s.t.} &  l_\nu \le x^\nu + p^\nu \le \min\left\{L_\nu, B - \sum\limits_{ \mu \in [N]\setminus \{ \nu \} } l_\mu \right\},\\
		& \sum_{\mu} (x^\mu + p^\mu) \le B
	\end{array}
\end{equation*}
    if $\nu \in [N_1]$, and
    \begin{equation*}
    \begin{array}{ll}
		\min\limits_{p^\nu \in \R} &  \left(\frac{1}{B} - \frac{S - x^\nu}{S^2} \right)p^\nu + \frac{S-x^\nu}{S^3}(p^\nu)^2 + \sum\limits_{\mu \in [N] \setminus \{ \nu\} }\frac{S-2x^\nu}{S^3}p^\nu p^\mu\\
		\text{s.t.} &  l_\nu \le x^\nu + p^\nu \le L_\nu
	\end{array}
\end{equation*}
 if $\nu \in [N] \setminus[N_1 ]$. When $[N_1] = [N]$, the result  follows directly from Theorem~\ref{thm-subproblem-solvability-jointly}. It remains to consider the case $[N_1] \neq [N]$. Suppose that there exists $p$ satisfying the private constraints but violating the shared constraint $\sum_\mu (x^\mu + p^\mu) \le B$. Then by \eqref{coro-subproblem-solvability-equation}, there exists $\hnu \in [N_1]$ such that 
 $$
    \begin{aligned}
      & \quad l_{\hnu} + \sum\limits_{\mu \in [N] \setminus \{ \hnu\} } (x^\mu + p^\mu) \\
      & \le   l_{\hnu} + \sum\limits_{ \mu\in [N]\setminus [N_1] } L_\mu + \sum\limits_{\mu \in [N_1] \setminus \{ \hnu\} }\min\left\{ \left( B - \sum\limits_{\omega \in [N] \setminus \{ \mu \}} l_\omega \right),L_\mu \right\}\\ 
      & \le   B.
    \end{aligned}
$$
Applying Theorem~\ref{thm-subproblem-general-solvability} guarantees the existence of a solution, which completes the proof.
 
\end{proof}

It is worth noting that, in practice, the upper bounds typically satisfy $L_\nu = +\infty$ (e.g., instance A14 in  Table~\ref{alg:table-time-2}), implying that the private strategy set of player $\nu$ is not compact. Nevertheless, the existence of a GNE is still guaranteed due to the presence of the shared constraints.

Under mild assumptions, the KKT system associated with Example~\ref{exm-internet} is always strongly regular at some point $(\bx,\blam)$. 

\begin{proposition}\label{prop:exm-strong-regularity}
Suppose that the parameters of Example~\ref{exm-internet} satisfy 
\begin{equation*}
\left\{\begin{array}{ll}
     L_\nu = +\infty\quad \text{for all } \nu \in [N_1],\ \text{and } L_\nu < +\infty\quad \text{for all } \nu \in [N]\setminus [N_1]\\
      \sum\limits_{\nu \in [N]} l_\nu < B, N_1 \ge 3 \\
      N_1^2 (l_\nu)^2 + (2a N_1+B-B N_1 ) l_\nu + a^2 - Ba \le 0 \quad \text{for all}\ \nu \in [N_1],\\
      \frac{1}{B} - \frac{s - l_\nu}{s^2} > 0 \quad \text{for all } \nu \in [N] \setminus [N_1],
\end{array}\right.
\end{equation*}
where $a = \sum_{\nu \in [N]\setminus [N_1]} l_\nu$, $b = \sqrt{ (2a N_1 + B -B N_1)^2 - 4N_1^2 (a^2-Ba)  }$, and $s = \sum_{\nu\in[N_1]} \bx^\nu + \sum_{\nu \in [N]\setminus [N_1]} l_\nu$. Then the KKT system associated with Example~\ref{exm-internet} is strongly regular at $(\bx,\blam)$ with
\begin{equation}\label{coro:exm-solution}
\begin{aligned}
\bx^\nu &= 
\begin{cases}
\frac{B N_1  -B -2 a N_1 + b}{2 N_1^2}, & \text{if } \nu \in [N_1], \\
l_\nu, & \text{otherwise, }
\end{cases}
\\
\blam^\nu_i &=
\begin{cases}
\frac{1}{B}-\frac{s-l_\nu}{s^2}, & \text{if } \nu \in [N]\setminus [N_1],\ i=1, \\
0, & \text{otherwise. }
\end{cases}
\end{aligned}
\end{equation}

\end{proposition}

\begin{proof}
It is straightforward to verify that \eqref{coro:exm-solution}
gives a KKT pair of Example~\ref{exm-internet}. It remains to show that it is strongly regular. Using the notations \eqref{index-partition-set}, we have
\begin{equation*}
    I^\nu_+ = 
    \begin{cases}
    \{1\}, & \text{if } \nu \in [N] \setminus [N_1],\\
    \emptyset, & \text{Otherwise,}
    \end{cases}
    \quad  I^\nu_- \cup I^\nu_+ = [m_{\nu}]\quad \text{for all } \nu \in [N].
\end{equation*}
Let $N_2 = N - N_1$, and define $M\in \R^{N\times N_1}$ by
$$
    M:= 
    \begin{pmatrix}
     s-2\bx^1 & \cdots & s-2\bx^{N_1} &\\
   \vdots & \vdots & \vdots \\
   s-2\bx^{1} & \cdots & s-2\bx^{N_1} 
    \end{pmatrix} + 
   \begin{pmatrix}
   s \mathbf{I}_{N_1}\\
   \mathbf{0}_{N_2 \times N_1} 
   \end{pmatrix},
$$
where $s = \sum_{\nu \in [N]} \bx^\nu$. By Theorem~\ref{thm:strong-regularity}, it suffices to show that 
\begin{equation*}
    \begin{pmatrix}
    \frac{M}{s^3} & \begin{array}{c}
    \mathbf{0}_{N_1 \times N_2}\\
    \mathbf{I}_{N_2}
    \end{array}
    \end{pmatrix} \dz = 0,
\end{equation*}
admits $\dz=0$ as its unique solution. Since 
\begin{equation*}
    {\rm det} \begin{pmatrix}
    \frac{M}{s^3} & \begin{array}{c}
    \mathbf{0}_{N_1 \times N_2}\\
    \mathbf{I}_{N_2}
    \end{array}
    \end{pmatrix}
     = \frac{1}{s^{2 N_1}} \left( 1 + \sum\limits_{\nu \in [N_1]} \frac{2\bx^\nu-s}{s} \right) \neq 0,
\end{equation*}
the desired result follows.
\end{proof}

Based on Proposition~\ref{prop:exm-strong-regularity}, instance A1 in Table~\ref{alg:table-time-1}, characterized by $N=10$, $N_1 = 9$, $B=1$, with $0.3 \le x^1 \le 0.5$ and $x^\nu \ge 0.01$ for all $\nu \neq 1$, admits a strongly regular solution given by
\begin{equation*}
\begin{aligned}
\bx^\nu &= 
\begin{cases}
\frac{1.3 + \sqrt{18.7}}{81} & \text{if } 2\le \nu \le 10, \\
0.3 & \text{if } \nu = 1,
\end{cases}
\quad
&\blam^\nu_i &=
\begin{cases}
1-\frac{s-0.3}{s^2} & \text{if } i = \nu=1, \\
0 & \text{otherwise. }
\end{cases}
\end{aligned}
\end{equation*}
In addition, the iterates generated by the SLCP method exhibit local quadratic convergence towards this solution (see Figure~\ref{fig:A1}), consistent with Theorem~\ref{thm:local quadratic convergence}.

\bibliographystyle{spmpsci}
\bibliography{references}

@article{han1977globally,
  title={A globally convergent method for nonlinear programming},
  author={Han, Shih-Ping},
  journal={J. Optim. Theory Appl.},
  volume={22},
  number={3},
  pages={297--309},
  year={1977},
  publisher={Springer}
}

@book{cottle2009linear,
  title={The Linear Complementarity Problem},
  author={Cottle, Richard W and Pang, Jong-Shi and Stone, Richard E},
  year={2009},
  publisher={SIAM},
  address = {Philadelphia}
}

@article{bonnans1994local,
  title={Local analysis of {Newton}-type methods for variational inequalities and nonlinear programming},
  author={Bonnans, J Fr{\'e}d{\'e}ric},
  journal={Appl. Math. Optim.},
  volume={29},
  number = {2},
  pages={161--186},
  year={1994},
  publisher={Springer}
}

@article{izmailov2015newton,
  title={{Newton}-type methods: A broader view},
  author={Izmailov, Alexey F and Solodov, Mikhail V},
  journal={J. Optim. Theory Appl.},
  volume={164},
  number={2},
  pages={577--620},
  year={2015},
  publisher={Springer}
}

@book{facchinei2003finite,
  title={Finite-Dimensional Variational Inequalities and Complementarity Problems, Part II},
  author={Facchinei, Francisco and Pang, Jong-Shi},
  year={2003},
  publisher={Springer},
  address ={New York}
}

@article{dontchev1996characterizations,
  title={Characterizations of strong regularity for variational inequalities over polyhedral convex sets},
  author={Dontchev, Asen L and Rockafellar, R Tyrrell},
  journal={SIAM J. Optim.},
  volume={6},
  number={4},
  pages={1087--1105},
  year={1996},
  publisher={SIAM}
}

@article{robinson1980strongly,
  title={Strongly regular generalized equations},
  author={Robinson, Stephen M},
  journal={Math. Oper. Res.},
  volume={5},
  number={1},
  pages={43--62},
  year={1980},
  publisher={INFORMS}
}

@article{facchinei2009generalized,
  title={Generalized {Nash} equilibrium problems and {Newton} methods},
  author={Facchinei, Francisco and Fischer, Andreas and Piccialli, Veronica},
  journal={Math. Program.},
  volume={117},
  number={1--2},
  pages={163--194},
  year={2009},
  publisher={Springer}
}

@article{ba2022exact,
  title={Exact penalization of generalized {Nash} equilibrium problems},
  author={Ba, Qin and Pang, Jong-Shi},
  journal={Oper. Res.},
  volume={70},
  number={3},
  pages={1448--1464},
  year={2022},
  publisher={INFORMS}
}

@book{palomar2010convex,
  title={Convex Optimization in Signal Processing and Communications},
  author={Palomar, Daniel P and Eldar, Yonina C},
  year={2010},
  publisher={Cambridge University Press},
  address = {New York, NY, USA}
}

@incollection{kesselman2005game,
  author    = {Kesselman, Alex and Leonardi, Stefano and Bonifaci, Vincenzo},
  title     = {Game-Theoretic Analysis of Internet Switching with Selfish Users},
  booktitle = {Internet and Network Economics},
  editor    = {Deng, Xiaotie and Ye, Yinyu},
  series    = {Lecture Notes in Computer Science},
  volume    = {3828},
  pages     = {236--245},
  year      = {2005},
  publisher = {Springer},
  address   = {Berlin, Heidelberg}
}

@incollection{hoffman2003approximate,
  title={On approximate solutions of systems of linear inequalities},
  author={Hoffman, Alan J},
  editor    = {Micchelli, Charles A.},
  booktitle={Selected Papers Of Alan J Hoffman: With Commentary},
  pages={174--176},
  year={2003},
  publisher={World Scientific},
  address = {Singapore}
}

@article{diao2025stability,
  title={Stability for {Nash} Equilibrium Problems},
  author={Diao, Ruoyu and Dai, Yu-Hong and Zhang, Liwei},
  journal={Math. Oper. Res.},
  year={2025},
  note={Published online},
  doi = {10.1287/moor.2024.0609},
  publisher={INFORMS}
}

@article{aubin1984lipschitz,
  title={Lipschitz behavior of solutions to convex minimization problems},
  author={Aubin, Jean-Pierre},
  journal={Math. Oper. Res.},
  volume={9},
  number={1},
  pages={87--111},
  year={1984},
  publisher={INFORMS}
}

@article{mordukhovich2007coderivative,
  title={Coderivative analysis of quasi-variational inequalities with applications to stability and optimization},
  author={Mordukhovich, Boris S and Outrata, Ji{\v{r}}{\'\i} V},
  journal={SIAM J. Optim.},
  volume={18},
  number={2},
  pages={389--412},
  year={2007},
  publisher={SIAM}
}

@article{monteiro1999potential,
  title={A potential reduction {Newton} method for constrained equations},
  author={Monteiro, Renato DC and Pang, Jong-Shi},
  journal={SIAM J. Optim.},
  volume={9},
  number={3},
  pages={729--754},
  year={1999},
  publisher={SIAM}
}

@article{dreves2011solution,
  title={On the solution of the {KKT} conditions of generalized {Nash} equilibrium problems},
  author={Dreves, Axel and Facchinei, Francisco and Kanzow, Christian and Sagratella, Simone},
  journal={SIAM J. Optim.},
  volume={21},
  number={3},
  pages={1082--1108},
  year={2011},
  publisher={SIAM}
}

@article{kanzow2016augmented,
  title={Augmented {Lagrangian} methods for the solution of generalized {Nash} equilibrium problems},
  author={Kanzow, Christian and Steck, Daniel},
  journal={SIAM J. Optim.},
  volume={26},
  number={4},
  pages={2034--2058},
  year={2016},
  publisher={SIAM}
}

@article{kanzow2018augmented,
  title={Augmented {Lagrangian} and exact penalty methods for quasi-variational inequalities},
  author={Kanzow, Christian and Steck, Daniel},
  journal={Comput. Optim. Appl.},
  volume={69},
  number={3},
  pages={801--824},
  year={2018},
  publisher={Springer}
}

@article{bueno2019optimality,
  title={Optimality conditions and constraint qualifications for generalized {Nash} equilibrium problems and their practical implications},
  author={Bueno, Luis Felipe and Haeser, Gabriel and Rojas, Frank Navarro},
  journal={SIAM J. Optim.},
  volume={29},
  number={1},
  pages={31--54},
  year={2019},
  publisher={SIAM}
}

@article{jordan2023first,
  title={First-order algorithms for nonlinear generalized {Nash} equilibrium problems},
  author={Jordan, Michael I and Lin, Tianyi and Zampetakis, Manolis},
  journal={J. Mach. Learn. Res.},
  volume={24},
  number={38},
  pages={1--46},
  year={2023}
}

@article{kim2023new,
  title={A new {Lagrangian}-based first-order method for nonconvex constrained optimization},
  author={Kim, Jong Gwang},
  journal={Oper. Res. Lett.},
  volume={51},
  number={3},
  pages={357--363},
  year={2023},
  publisher={Elsevier}
}

@article{facchinei2010penalty,
  title={Penalty methods for the solution of generalized {Nash} equilibrium problems},
  author={Facchinei, Francisco and Kanzow, Christian},
  journal={SIAM J. Optim.},
  volume={20},
  number={5},
  pages={2228--2253},
  year={2010},
  publisher={SIAM}
}

@article{fischer1995newton,
  title={A {Newton}-type method for positive-semidefinite linear complementarity problems},
  author={Fischer, Andreas},
  journal={J. Optim. Theory Appl.},
  volume={86},
  number={3},
  pages={585--608},
  year={1995},
  publisher={Springer}
}

@book{nocedal2006numerical,
  title={Numerical Optimization},
  author={Nocedal, Jorge and Wright, Stephen J},
  year={2006},
  publisher={Springer},
  address={New York}
}

@article{schiro2013solution,
  title={On the solution of affine generalized {Nash} equilibrium problems with shared constraints by {Lemke}'s method},
  author={Schiro, Dane A and Pang, Jong-Shi and Shanbhag, Uday V},
  journal={Math. Program.},
  volume={142},
  number={1--2},
  pages={1--46},
  year={2013},
  publisher={Springer}
}

@article{dreves2014finding,
  title={Finding all solutions of affine generalized {Nash} equilibrium problems with one-dimensional strategy sets},
  author={Dreves, Axel},
  journal={Math. Methods Oper. Res.},
  volume={80},
  number={2},
  pages={139--159},
  year={2014},
  publisher={Springer}
}

@article{dolan2002benchmarking,
  title={Benchmarking optimization software with performance profiles},
  author={Dolan, Elizabeth D and Mor{\'e}, Jorge J},
  journal={Math. Program.},
  volume={91},
  number={2},
  pages={201--213},
  year={2002},
  publisher={Springer}
}

@article{gill2005snopt,
  title={{SNOPT}: An {SQP} algorithm for large-scale constrained optimization},
  author={Gill, Philip E and Murray, Walter and Saunders, Michael A},
  journal={SIAM Rev.},
  volume={47},
  number={1},
  pages={99--131},
  year={2005},
  publisher={SIAM}
}

@article{schittkowski1986nlpql,
  title={{NLPQL}: A {FORTRAN} subroutine solving constrained nonlinear programming problems},
  author={Schittkowski, Klaus},
  journal={Ann. Oper. Res.},
  volume={5},
  number={2},
  pages={485--500},
  year={1986},
  publisher={Springer}
}

@article{murtagh1978large,
  title={Large-scale linearly constrained optimization},
  author={Murtagh, Bruce A and Saunders, Michael A},
  journal={Math. Program.},
  volume={14},
  number={1},
  pages={41--72},
  year={1978},
  publisher={Springer}
}

@article{drud1985conopt,
  title={{CONOPT}: A {GRG} code for large sparse dynamic nonlinear optimization problems},
  author={Drud, Arne},
  journal={Math. Program.},
  volume={31},
  number={2},
  pages={153--191},
  year={1985},
  publisher={Springer}
}

@article{arrow1954existence,
title = {Existence of an equilibrium for a competitive economy},
author = {Arrow, Kenneth J and Debreu, Gerard},
journal = {Econometrica},
volume = {22},
number = {3},
pages = {265--290},
year = {1954}
}

@article{facchinei2010generalized,
  title={Generalized {Nash} equilibrium problems},
  author={Facchinei, Francisco and Kanzow, Christian},
  journal={Ann. Oper. Res.},
  volume={175},
  number={1},
  pages={177--211},
  year={2010},
  publisher={Springer}
}

@article{facchinei2007generalized,
  title={Generalized {Nash} equilibrium problems},
  author={Facchinei, Francisco and Kanzow, Christian},
  journal={4OR},
  volume={5},
  number={3},
  pages={173--210},
  year={2007},
  publisher={Springer}
}

@article{burke2014sequential,
  title={A sequential quadratic optimization algorithm with rapid infeasibility detection},
  author={Burke, James V and Curtis, Frank E and Wang, Hao},
  journal={SIAM J. Optim.},
  volume={24},
  number={2},
  pages={839--872},
  year={2014},
  publisher={SIAM}
}

@article{de1996semismooth,
  title={A semismooth equation approach to the solution of nonlinear complementarity problems},
  author={De Luca, Tecla and Facchinei, Francisco and Kanzow, Christian},
  journal={Math. Program.},
  volume={75},
  number={3},
  pages={407--439},
  year={1996},
  publisher={Springer}
}

@article{de2000theoretical,
  title={A theoretical and numerical comparison of some semismooth algorithms for complementarity problems},
  author={De Luca, Tecla and Facchinei, Francisco and Kanzow, Christian},
  journal={Comput. Optim. Appl.},
  volume={16},
  number={2},
  pages={173--205},
  year={2000},
  publisher={Springer}
}

@article{debreu1952social,
  title={A social equilibrium existence theorem},
  author={Debreu, Gerard},
  journal={Proc. Natl. Acad. Sci.},
  volume={38},
  number={10},
  pages={886--893},
  year={1952},
  publisher={National Academy of Sciences}
}

@article{zhou2005generalized,
  title={The generalized {Nash} equilibrium model for oligopolistic transit market with elastic demand},
  author={Zhou, Jing and Lam, William HK and Heydecker, Benjamin G},
  journal={Transp. Res. Part B: Methodol.},
  volume={39},
  number={6},
  pages={519--544},
  year={2005},
  publisher={Elsevier}
}

@article{xiao2007competition,
  title={Competition and efficiency of private toll roads},
  author={Xiao, Feng and Yang, Hai and Han, Deren},
  journal={Transp. Res. Part B: Methodol.},
  volume={41},
  number={3},
  pages={292--308},
  year={2007},
  publisher={Elsevier}
}

@article{stein2018noncooperative,
  title={The noncooperative transportation problem and linear generalized {Nash} games},
  author={Stein, Oliver and Sudermann-Merx, Nathan},
  journal={Eur. J. Oper. Res.},
  volume={266},
  number={2},
  pages={543--553},
  year={2018},
  publisher={Elsevier}
}

@article{ban2019general,
  title={A general equilibrium model for transportation systems with e-hailing services and flow congestion},
  author={Ban, Xuegang Jeff and Dessouky, Maged and Pang, Jong-Shi and Fan, Rong},
  journal={Transp. Res. Part B: Methodol.},
  volume={129},
  pages={273--304},
  year={2019},
  publisher={Elsevier}
}

@article{pang2008distributed,
  title={Distributed power allocation with rate constraints in {Gaussian} parallel interference channels},
  author={Pang, Jong-Shi and Scutari, Gesualdo and Facchinei, Francisco and Wang, Chaoxiong},
  journal={IEEE Trans. Inf. Theory},
  volume={54},
  number={8},
  pages={3471--3489},
  year={2008},
  publisher={IEEE}
}

@article{krawczyk2005coupled,
  title={Coupled constraint {Nash} equilibria in environmental games},
  author={Krawczyk, Jacek B},
  journal={Resour. Energy Econ.},
  volume={27},
  number={2},
  pages={157--181},
  year={2005},
  publisher={Elsevier}
}

@article{breton2006game,
  title={A game-theoretic formulation of joint implementation of environmental projects},
  author={Breton, Michele and Zaccour, Georges and Zahaf, Mehdi},
  journal={Eur. J. Oper. Res.},
  volume={168},
  number={1},
  pages={221--239},
  year={2006},
  publisher={Elsevier}
}

@article{fischer2014generalized,
  title={Generalized {Nash} equilibrium problems - recent advances and challenges},
  author={Fischer, Andreas and Herrich, Markus and Sch{\"o}nefeld, Klaus},
  journal={Pesq. Oper.},
  volume={34},
  number={3},
  pages={521--558},
  year={2014},
  publisher={SciELO Brasil}
}

@article{pang2005quasi,
  title={Quasi-variational inequalities, generalized {Nash} equilibria, and multi-leader-follower games},
  author={Pang, Jong-Shi and Fukushima, Masao},
  journal={Comput. Manag. Sci.},
  volume={2},
  number={1},
  pages={21--56},
  year={2005},
  publisher={Springer}
}

@article{facchinei2011partial,
  title={Partial penalization for the solution of generalized {Nash} equilibrium problems},
  author={Facchinei, Francisco and Lampariello, Lorenzo},
  journal={J. Global Optim.},
  volume={50},
  number={1},
  pages={39--57},
  year={2011},
  publisher={Springer}
}

@article{fukushima2011restricted,
  title={Restricted generalized {Nash} equilibria and controlled penalty algorithm},
  author={Fukushima, Masao},
  journal={Comput. Manag. Sci.},
  volume={8},
  number={3},
  pages={201--218},
  year={2011},
  publisher={Springer}
}

@article{kanzow2019quasi,
  title={Quasi-variational inequalities in {Banach} spaces: theory and augmented {Lagrangian} methods},
  author={Kanzow, Christian and Steck, Daniel},
  journal={SIAM J. Optim.},
  volume={29},
  number={4},
  pages={3174--3200},
  year={2019},
  publisher={SIAM}
}

@article{fischer2016globally,
  title={A globally convergent {LP}-{Newton} method},
  author={Fischer, Andreas and Herrich, Markus and Izmailov, Alexey F and Solodov, Mikhail V},
  journal={SIAM J. Optim.},
  volume={26},
  number={4},
  pages={2012--2033},
  year={2016},
  publisher={SIAM}
}

@article{dreves2014new,
  title={A new error bound result for generalized {Nash} equilibrium problems and its algorithmic application},
  author={Dreves, Axel and Facchinei, Francisco and Fischer, Andreas and Herrich, Markus},
  journal={Comput. Optim. Appl.},
  volume={59},
  number={1},
  pages={63--84},
  year={2014},
  publisher={Springer}
}

@article{mathiesen1987algorithm,
  title={An algorithm based on a sequence of linear complementarity problems applied to a {Walrasian} equilibrium model: An example},
  author={Mathiesen, Lars},
  journal={Math. Program.},
  volume={37},
  number={1},
  pages={1--18},
  year={1987},
  publisher={Springer}
}

@article{mathiesen1985computational,
  title={Computational experience in solving equilibrium models by a sequence of linear complementarity problems},
  author={Mathiesen, Lars},
  journal={Oper. Res.},
  volume={33},
  number={6},
  pages={1225--1250},
  year={1985},
  publisher={INFORMS}
}

@article{wang2025symmetric,
  author  = {Wang, Hailing and Wu, Di and Wang, Song and Teo, Kok Lay and Yu, Changjun},
  title   = {A symmetric {Gauss--Seidel}-based majorized augmented {Lagrangian} method for generalized {Nash} equilibrium problems in {Hilbert} spaces},
  journal = {SIAM J. Optim.},
  volume  = {35},
  number  = {3},
  pages   = {2040--2065},
  year    = {2025}
}

@article{kanzow2019multiplier,
  author  = {Kanzow, Christian and Karl, Veronika and Steck, Daniel and Wachsmuth, Daniel},
  title   = {The multiplier-penalty method for generalized {Nash} equilibrium problems in {Banach} spaces},
  journal = {SIAM J. Optim.},
  volume  = {29},
  number  = {1},
  pages   = {767--793},
  year    = {2019}
}

@incollection{mathiesen1985computation,
  author    = {Mathiesen, Lars},
  title     = {Computation of economic equilibria by a sequence of linear complementarity problems},
  booktitle = {Economic Equilibrium: Model Formulation and Solution},
  editor    = {Manne, Alan S.},
  series    = {Mathematical Programming Studies},
  volume    = {23},
  pages     = {144--162},
  year      = {1985},
  publisher = {Springer},
  address   = {Berlin, Heidelberg}
}

@book{rockafellar1998variational,
  title     = {Variational Analysis},
  author    = {Rockafellar, R. Tyrrell and Wets, Roger J.-B.},
  year      = {1998},
  publisher = {Springer},
  address   = {Berlin, Heidelberg}
}

@book{rockafellar1970convex,
  title     = {Convex Analysis},
  author    = {Rockafellar, R. Tyrrell},
  year      = {1970},
  publisher = {Princeton University Press},
  address   = {Princeton, NJ}
}

@incollection{dontchev1997characterizations,
  title     = {Characterizations of Lipschitzian Stability in Nonlinear Programming},
  author    = {Dontchev, Asen L. and Rockafellar, R. Tyrrell},
  booktitle = {Mathematical Programming with Data Perturbations},
  editor    = {Fiacco, Anthony V.},
  pages     = {65--82},
  year      = {1997},
  publisher = {Marcel Dekker},
  address   = {New York}
}

@book{dontchev2009implicit,
  title     = {Implicit Functions and Solution Mappings: A View from Variational Analysis},
  author    = {Dontchev, Asen L. and Rockafellar, R. Tyrrell},
  year      = {2009},
  publisher = {Springer},
  address   = {New York}
}

@incollection{murtagh1982projected,
  title     = {A Projected {Lagrangian} Algorithm and Its Implementation for Sparse Nonlinear Constraints},
  author    = {Murtagh, Bruce A. and Saunders, Michael A.},
  booktitle = {Algorithms for Constrained Minimization of Smooth Nonlinear Functions},
  editor    = {Buckley, A. G. and Goffin, J.-L.},
  series    = {Mathematical Programming Studies},
  volume    = {16},
  pages     = {84--117},
  year      = {1982},
  publisher = {North-Holland},
  address   = {Amsterdam}
}

@book{Mangasarian1969,
  author    = {Mangasarian, Olvi L.},
  title     = {Nonlinear Programming},
  publisher = {McGraw-Hill},
  address   = {New York},
  year      = {1969}
}

%\begin{acknowledgements}
%\end{acknowledgements}



% \begin{thebibliography}{}
% %
% % and use \bibitem to create references. Consult the Instructions
% % for authors for reference list style.
% % Format for Journal Reference
% Author, Article title, Journal, Volume, page numbers (year)
% Author, Book title, page numbers. Publisher, place (year)
% \end{thebibliography}

\end{document}